\documentclass[review,hidelinks,onefignum,onetabnum]{siamart220329}

\usepackage{lipsum}
\usepackage{amsfonts}
\usepackage{graphicx}
\usepackage{epstopdf}
\usepackage{algorithmic}
\ifpdf
\DeclareGraphicsExtensions{.eps,.pdf,.png,.jpg}
\else
\DeclareGraphicsExtensions{.eps}
\fi

\newsiamremark{remark}{Remark}
\newsiamremark{hypothesis}{Hypothesis}
\crefname{hypothesis}{Hypothesis}{Hypotheses}
\newsiamthm{claim}{Claim}

\headers{Accelerated preconditioned ADMM}{}

\title{Accelerating preconditioned ADMM via degenerate proximal point mappings\thanks{Submitted to the editors DATE.
		\funding{The work of Defeng Sun was supported by RGC Senior Research Fellow Scheme No. SRFS2223-5S02 and  GRF Project No. 15304721. The work of Yancheng Yuan was supported by the Research Center for Intelligent Operations Research and The Hong Kong Polytechnic University under grants P0038284 and P0045485. The work of Guojun Zhang was supported in part by the Huawei Collaborative Grant ``Fast Optimal Transport Algorithms for Communications and Networking" with UGC's matching fund. The work of Xinyuan Zhao was supported in part by the National Natural Science Foundation of China under Project No. 12271015 and the Department of Applied Mathematics, the Hong Kong Polytechnic University, where this project was initiated during her visit in the summer of 2023.}}}

\author{Defeng Sun\thanks{Department of Applied Mathematics, The Hong Kong Polytechnic University, Hung Hom, Hong Kong
		(Corresponding author. \email{defeng.sun@polyu.edu.hk}).}
	\and Yancheng Yuan\thanks{Department of Applied Mathematics, The Hong Kong Polytechnic University, Hung Hom, Hong Kong
		(\email{yancheng.yuan@polyu.edu.hk}).}
	\and Guojun Zhang\thanks{Department of Applied Mathematics, The Hong Kong Polytechnic University, Hung Hom, Hong Kong
		(\email{guojun.zhang@connect.polyu.hk}).}
	\and Xinyuan Zhao\thanks{Department of Mathematics, Beijing University of Technology, Beijing, P.R. China.
		(\email{xyzhao@bjut.edu.cn}).} }

\usepackage{amsopn}

\ifpdf
\hypersetup{
	pdftitle={An Example Article3333},
	pdfauthor={}
}
\fi

\usepackage{varwidth}
\newtheorem{assumption}{Assumption}

\def\hcT{{\widehat { \cal T}}}
\def\hcQ{{\widehat { \cal Q}}}
\def\hcF{{\widehat { \cal F}}}

\def\tcT{\widetilde{\mathcal{T}}}
\def\tcF{\widetilde{\mathcal{F}}}

\def\cM{{\cal M}}
\def\cC{{\cal C}}
\def\cT{{\cal T}}

\def\cH{{\cal H}}
\def\cU{{\cal U}}

\def\cI{{\cal I}}

\def\hy{{\hat y}}
\def\hz{{\hat z}}
\def\hx{{\hat x}}
\def\hw{{\hat w}}
\def\by{{\bar y}}
\def\bz{{\bar z}}
\def\bx{{\bar x}}
\def\tx{{\widetilde x}}
\def\bw{{\bar w}}
\def\bu{{\bar u}}

\begin{document}
	\maketitle
	\begin{abstract}
		In this paper, we aim to accelerate a preconditioned alternating direction method of multipliers (pADMM), whose proximal terms are convex quadratic functions, for solving linearly constrained convex optimization problems. To achieve this, we first reformulate the pADMM into a form of proximal point method (PPM) with a positive semidefinite preconditioner which can be degenerate due to the lack of strong convexity of the proximal terms in the pADMM. Then we accelerate the pADMM by accelerating the reformulated degenerate PPM (dPPM). Specifically, we first propose an accelerated dPPM by integrating the Halpern iteration and the fast Krasnosel'ski\u{i}-Mann iteration into it, achieving asymptotic $o(1/k)$ and non-asymptotic $O(1/k)$ convergence rates. Subsequently, building upon the accelerated dPPM, we develop an accelerated pADMM algorithm that exhibits both asymptotic $o(1/k)$ and non-asymptotic $O(1/k)$ nonergodic convergence rates concerning the {Karush–Kuhn–Tucker} residual and the primal objective function value gap. Preliminary numerical experiments validate the theoretical findings, demonstrating that the accelerated pADMM outperforms the pADMM in solving convex quadratic programming problems.
	\end{abstract}
	
	
	
	\section{Introduction}
	Let $\mathbb{X}$, $\mathbb{Y}$, and $\mathbb{Z}$ be three finite-dimensional real Euclidean spaces, each equipped with an inner product $\langle \cdot, \cdot \rangle$ and its induced norm $\|\cdot\|$. In this paper, we aim  to settle the issue on how to accelerate an alternating direction method of multipliers (ADMM) with semi-proximal terms for solving the following convex optimization problem:
	\begin{equation}\label{primal}
		\begin{array}{cc}
			\min _{y \in \mathbb{Y} , z \in \mathbb{Z}} & f_{1}(y)+f_{2}(z)\\
			\text{subject to}& {B}_{1}y+{B}_{2}z=c,
		\end{array}
	\end{equation}
	where $f_{1}: \mathbb{Y} \to (-\infty, +\infty]$ and $f_{2}: \mathbb{Z} \to (-\infty, +\infty]$ are two proper closed convex functions, 
	$B_1: \mathbb{Y} \to \mathbb{X}$ and $B_2: \mathbb{Z} \to \mathbb{X}$ are two given linear operators, and $c\in \mathbb{X}$ is a given point. {For a linear operator $B:\mathbb{X}\to\mathbb{X}$, we define its norm as $\|B\|:=\operatorname{sup}_{\|x\|\leq 1} \|Bx\|$. } Moreover, for any convex function $f: \mathbb{X}\to (-\infty, +\infty]$, we use  $\operatorname{dom}(f):=\{x \in \mathbb{X}: f(x)<\infty\}$  to denote  its effective domain,  $f^*: \mathbb{X} \to (-\infty,+\infty]$ to represent its Fenchel conjugate, and  ${\rm Prox}_{f}(\cdot)$ to denote its associated Moreau-Yosida proximal mapping  \cite{rockafellar1998variational}, respectively. Let $\sigma >0$ be a given penalty parameter. The augmented Lagrangian function of problem \eqref{primal} is defined by, for any  $(y,z,x) \in \mathbb{Y}\times \mathbb{Z}\times \mathbb{X}$,
	\[
	L_{\sigma}(y, z; x) := f_1(y) + f_2(z) + \langle x, B_1 y + B_2 z - c \rangle+\frac{\sigma}{2}\|B_1 y + B_2 z - c\|^2.
	\]
	The {dual of problem   \eqref{primal} is given by}
	\begin{equation}\label{dual}
		\max _{x \in \mathbb{X}}\left\{-f_{1}^{*}(-B_{1}^* x)-f_{2}^{*}(-B_{2}^* x)-\langle c, x\rangle\right\},
	\end{equation}
	where $B_1^*: \mathbb{X} \to \mathbb{Y}$ and $B_2^*: \mathbb{X} \to \mathbb{Z}$ are the adjoint of $B_1$ and $B_2$, respectively. For ease of notation, let $w:=(y, z, x)$ and $\mathbb{W}:=\mathbb{Y}\times \mathbb{Z}\times \mathbb{X}$. For any self-adjoint positive semidefinite linear operator $\cM: \mathbb{X} \to \mathbb{X}$, denote $\|x\|_{\cM}:=\sqrt{\langle x, \cM x\rangle}$. In extending the framework of ADMM with semi-proximal terms and larger dual step lengths studied in \cite{fazel2013hankel} to cover the case of generalized ADMM in the sense of {Eckstein} and  Bertsekas \cite{eckstein1992douglas}, Xiao et al.  \cite{xiao2018generalized} considered the following preconditioned ADMM (pADMM) {for solving problem \eqref{primal}:}
	\begin{algorithm}[ht]
		\caption{A pADMM for solving two-block convex optimization problem \eqref{primal}}
		\label{alg:pADMM}
		\begin{algorithmic}[1]
			\STATE {Input: Let $\cT_1$ and $\cT_2$ be two self-adjoint positive semidefinite linear operators on $\mathbb{Y}$ and $\mathbb{Z}$, respectively. Choose an initial point $w^{0}=(y^{0}, z^{0}, x^{0})\in \operatorname{dom} (f_1) \times \operatorname{dom} (f_2)\times \mathbb{X}$. Set parameters $\sigma > 0$ and $\rho_k \in (0, 2]$ for any $k\ge 0$. For $k=0,1, \ldots,$ perform the following steps in each iteration.}
			\STATE {Step 1. $\bz^{k}=\underset{z \in \mathbb{Z}}{\arg \min }\left\{L_\sigma\left(y^k, z ; x^k\right)+\frac{1}{2}\|z-z^{k}\|_{\mathcal{T}_2}^2\right\}$.}
			\STATE{Step 2. $\bx^{k}={x}^k+\sigma (B_{1}{y}^{k}+B_{2}\bz^{k}-c) $.}
			\STATE {Step 3. $\by^{k}=\underset{y \in \mathbb{Y}}{\arg \min }\left\{L_\sigma\left(y, \bz^{k} ; \bx^{k}\right)+\frac{1}{2}\|y-y^{k}\|_{\mathcal{T}_1}^2\right \}$.
			}
			\STATE {Step 4. $w^{k+1}= (1-\rho_k){w}^{k}+ \rho_k \bw^{k}$.}
		\end{algorithmic}
	\end{algorithm}
	
	Since the two linear operators $\mathcal{T}_1$ and $\mathcal{T}_2$  in {Algorithm \ref{alg:pADMM} are only assumed to be positive semidefinite}, the pADMM framework includes ADMM \cite{gabay1976dual,glowinski1975approximation}, proximal ADMM \cite{eckstein1994some}, and semi-proximal ADMM \cite{fazel2013hankel} as special cases. Specifically, when $\mathcal{T}_1=0$, $\mathcal{T}_2=0$, and $\rho_k \equiv 1$, pADMM corresponds to the ADMM \cite{gabay1976dual,glowinski1975approximation} with a unit dual step-length. Additionally, if both $\mathcal{T}_1$ and $\mathcal{T}_2$ are positive definite, then the above pADMM scheme with $\rho_k \equiv 1$ reduces to the proximal ADMM proposed by Eckstein \cite{eckstein1994some} with {possible changes of inner products}. Moreover, when $\rho_k \equiv 1$, pADMM corresponds to a special case studied in Fazel et al. \cite{fazel2013hankel}   {under the setting of semi-proximal ADMM}. On the other hand, it is worth highlighting that pADMM is capable of handling multi-block convex optimization problems by introducing the symmetric Gauss-Seidel (sGS) operator \cite{li2016schur,li2019block} as a linear operator in the proximal term, which has proven to be effective in solving large-scale optimization problems \cite{li2016schur,li2018qsdpnal,liang2022qppal,xiao2018generalized}. Therefore, given the broad applicability and effectiveness of the pADMM framework, investigating the acceleration of pADMM is of significant interest.
	
	Before discussing {the} acceleration techniques, we mention several developments in the convergence rate analysis of pADMM that are relevant to our paper: Monteiro and Svaiter \cite{monteiro2013iteration} first established an ergodic $O(1/k)$  convergence rate for the ADMM with a unit dual stepsize in terms of the Karush-Kuhn-Tucker (KKT) type residual; Davis and Yin \cite{davis2016convergence} provided a nonergodic iteration complexity bound of $o(1/\sqrt{k})$ for the ADMM with a unit dual stepsize, focusing on primal feasibility violations and the primal objective function value gap; and  Cui et al. \cite{cui2016convergence} demonstrated that a majorized ADMM, including the classical ADMM, exhibits a nonergodic $O(1/\sqrt{k})$  convergence rate concerning the KKT optimality condition.
	
	To further enhance the efficiency of the pADMM, researchers have explored two main approaches for accelerations. One approach involves integrating Nesterov's extrapolation \cite{beck2009fast,nesterov1983method} directly into the pADMM to develop accelerated variants. For instance,  when one of the objective functions is strongly convex, Xu \cite{xu2017accelerated} proposed an accelerated linearized ADMM (LADMM) with an ergodic $O(1/k^2)$ convergence rate concerning feasibility violations and objective function values. For more results on the strongly convex case, {one may} refer to \cite{goldstein2014fast,tang2023self} { and the references therein}. On the other hand, {in the absence of strong convexity assumption}, which is a primary focus of this paper, several accelerated ADMM {versions} with a convergence rate of $O(1/k)$ have been introduced. Specifically, assuming $f_2(\cdot)$ {to be a} smooth convex function with $L_{f_2}$-Lipschitz continuous gradient, Ouyang et al. \cite{ouyang2015accelerated} proposed an accelerated variant of LADMM. In terms of {function values} and feasibility violations, the ergodic convergence rate of this method associated with the $L_{f_2}$ part surpasses $O(1/k)$, while the ergodic rate for other parts remains $O(1/k)$ (see Tables 2 and 3 in \cite{ouyang2015accelerated}). {Turning to the nonergodic rates along this approach, two notable works are done by Li and Lin \cite{li2019accelerated} and Sabach and Teboulle \cite{sabach2022faster}:}
	
	In \cite{li2019accelerated}, Li and Lin modified the accelerated LADMM proposed in \cite{ouyang2015accelerated} to obtain a nonergodic $O(1/(1+ k(1-\tau)))$ convergence rate in terms of function values and feasibility violations {with the dual step length $\tau$ to be restricted in   $  (0.5,1)$.} In their algorithm, {the largest eigenvalues  $\lambda_{\max}(\cT_i^k)$ of $\cT_i^k$ for $i=1,2$ vary with $k$ and satisfy}
	$$
	\lambda_{\max}(\cT_1^k) \sim O(1+k(1-\tau)),\quad  \lambda_{\max}(\cT_2^k) \sim O(1+k(1-\tau))
	$$
	at the $k$-th iteration, as indicated in formulas (6a) and (6b) of \cite{li2019accelerated}. This implies that the primal step length approaches zero as $k$ tends to infinity. Furthermore, when the dual step length $\tau$ approaches one, the $O(1/(1+ k(1-\tau)))$ complexity result will be lost (see Theorem 1 of \cite{li2019accelerated}).
	{Different from Li and Lin \cite{li2019accelerated}}, in \cite{sabach2022faster}, Sabach and Teboulle introduced an accelerated variant of pADMM with a nonergodic $O(1/k)$ convergence rate in terms of function values and feasibility violations when $\mathcal{T}_2$ is positive definite. {Note that the dual step length $\mu$ in Sabach and Teboulle's algorithm is required to satisfy}
	$$
	\mu \in (0,\delta], \quad \delta=1- \frac{  \sigma \lambda_{\max }\left(B_{2}^* B_{2}\right)}{\left(\sigma \lambda_{\max }\left(B_{2}^* B_{2}\right)+\lambda_{\min }\left(\cT_2\right)\right)} <1
	$$
	{according to} Lemma 5.7 of \cite{sabach2022faster}, where $\lambda_{\min}(\cT_2)$ is the smallest eigenvalue of $\cT_2$. When $\mathcal{T}_2$ tends to {be degenerate, i.e.,
		$\lambda_{\min}(\cT_2)$} goes to zero, the dual step length $\mu$ approaches zero as $\delta$ tends to zero, {which implies that the obtained complexity bounds} blow up  (see formulas (4.11) and (4.12) in Theorem 4.5 of \cite{sabach2022faster}). {Thus, to the best of our knowledge, there is still a big gap in using this acceleration technique to handle the case where both $\mathcal{T}_1$ and $\mathcal{T}_2$ are positive semidefinite and/or with large step lengths.}
	
	Another approach for accelerating the pADMM is first to reformulate it as a fixed-point iterative method, if possible, and then to accelerate the pADMM by accelerating the fixed-point iterative method. For instance, Kim \cite{kim2021accelerated} introduced an accelerated proximal point method (PPM) \cite{rockafellar1976augmented,rockafellar1976monotone} {for the purpose of
		achieving a fast rate of $O(1/k)$.}
	By reformulating the ADMM as the Douglas-Rachford (DR) splitting method,  {which is a special case of the PPM \cite{eckstein1992douglas}}, Kim \cite{kim2021accelerated} obtained an accelerated ADMM {that possesses} a nonergodic $O(1/k)$ convergence rate concerning the primal feasibility violations. Since Contreras and Cominetti \cite{contreras2023optimal} established a close connection between Kim's accelerated PPM and Halpern's iteration \cite{halpern1967fixed,lieder2021convergence,sabach2017first}, it is natural to employ the Halpern iteration to obtain an accelerated ADMM, owing to the simplicity of the Halpern iteration. Following this line, Tran-Dinh and Luo \cite{tran2021halpern} {obtained} an accelerated ADMM by accelerating {the DR splitting method}, which achieves a nonergodic $O(1/k)$ convergence rate on the forward-backward residual operator associated with the dual problem \eqref{dual}. When $f_2(\cdot)$ is a strongly convex function, this result actually represents the $O(1/k)$ convergence rate in terms of primal feasibility violations. Furthermore, by using a more effective Peaceman-Rachford splitting method in general, Zhang et al. \cite{zhang2022efficient} proposed a Halpern-Peaceman-Rachford method, which exhibits a nonergodic convergence rate of $O(1/k)$ concerning both the KKT residual and the primal objective function value gap. {Recently, assuming {that} both $\cT_1$ and $\cT_2$ are positive definite, Yang et al. \cite{yang2023accelerated} reformulated the pADMM as a preconditioned PPM \cite{bonnans1995family,li2020asymptotically}. By applying the Halpern iteration to this preconditioned PPM, they proposed an accelerated pADMM with  {the} relaxation factor $\rho\in(0,2]$, which enjoys a nonergodic $O(1 / k)$ convergence rate for the fixed-point KKT residual.} While the accelerated pADMM variants mentioned above have demonstrated success in specific applications, it is worth emphasizing that these methods cannot cover the general case of pADMM with semi-proximal terms.
	
	{Compared to {the first approach of}  employing Nesterov's extrapolation to obtain an accelerated pADMM, the second approach mentioned above does not impose restrictive requirements on the step lengths. This {inspires}  us to further explore accelerating the pADMM with semi-proximal terms by focusing on the second approach}. More recently, Bredies et al. \cite{bredies2022degenerate} extended their earlier works \cite{bredies2015preconditioned, bredies2017proximal} by introducing the degenerate PPM (dPPM) with a positive semidefinite preconditioner. They regarded the Chambolle-Pock scheme \cite{chambolle2011first} under the condition of $\tau \sigma \|L\|^2=1$ (see formula (3.3) in \cite{bredies2022degenerate}) as a dPPM to discuss its convergence  \cite{condat2013primal}. Note that the Chambolle-Pock scheme under the condition of $\tau \sigma \|L\|^2\leq 1$ is {actually equivalent to LADMM}, and the convergence properties of LADMM, even with larger dual step lengths in the interval $(0,(1+\sqrt{5})/2)$, have already been covered in \cite{fazel2013hankel} {under a much more general setting of ADMM with semi-proximal terms}.  {So it is not of absolute necessity to study the dPPM if it is only used to analyze the convergence of the  Chambolle-Pock scheme \cite{chambolle2011first} with the condition $\tau \sigma \|L\|^2=1$.} However, the work of \cite{bredies2022degenerate} has motivated us to look at the pADMM from the perspective of dPPM {for extending the work of Yang et al. \cite{yang2023accelerated} where both  $\cT_1$ and $\cT_2$ are assumed to be positive definite.} Indeed, we establish an equivalence between the pADMM and the dPPM. Consequently, one may consider employing the Halpern iteration to accelerate the dPPM, thus obtaining an accelerated pADMM. On the other hand, it is worth noting that Contreras and Cominetti \cite{contreras2023optimal} demonstrated that the best possible convergence rate for general Mann iterations in normed spaces, including the Halpern iteration, is lower bounded by $O(1/k)$. In contrast, inspired by the second-order dynamical system with a vanishing damping term proposed in \cite{bot2023fastogda} for solving monotone equations, Bot and Nguyen \cite{bot2023fast} introduced the fast Krasnosel'ski\u{i}-Mann (KM) iteration with asymptotic $o(1/k)$ convergence rates for finding a fixed-point of the nonexpansive operator, which appears to offer better convergence rates than $O(1/k)$ in certain applications. Therefore, in this paper, we integrate the fast KM iteration and the Halpern iteration into the dPPM to accelerate it, aiming to achieve both asymptotic $o(1/k)$ and non-asymptotic $O(1/k)$ convergence rates. The main contributions of this paper can be highlighted as follows:
	\begin{enumerate}
		\item We propose a globally convergent accelerated dPPM by unifying the Halpern iteration and the fast KM iteration. This method exhibits both asymptotic $o(1/k)$ and non-asymptotic $O(1/k)$ in terms of the operator residual.
		\item We establish the equivalence between the pADMM and the dPPM. Utilizing the accelerated dPPM, we introduce a globally convergent accelerated pADMM, which enjoys both asymptotic $o(1/k)$ and non-asymptotic $O(1/k)$ nonergodic convergence rates concerning the KKT residual and the primal objective function value gap. In our accelerated pADMM, both $\mathcal{T}_1$ and $\mathcal{T}_2$ can be positive semidefinite under mild conditions. Additionally, the relaxation factor $\rho$ can be chosen in the interval $(0,2]$, where a larger value of $\rho$ generally leads to better performance.
		\item We implement the proposed accelerated pADMM and the pADMM to solve convex quadratic programming (QP) problems. Our {preliminary} numerical results exhibit the superiority of the proposed accelerated pADMM over the pADMM.
	\end{enumerate}
	
	The remaining parts of this paper are organized as follows. In Section \ref{Sec: acc dPPM}, we briefly introduce the dPPM and propose an accelerated dPPM. In Section \ref{Sec: acc-pADMM}, we first establish the equivalence between the pADMM and the dPPM, and then present an accelerated pADMM based on the accelerated dPPM. Section \ref{Sec: Numerical} provides some numerical results to show the superiority of the accelerated pADMM over the pADMM by using the convex {QP} problem as an illustrative example. Finally, we conclude the paper in Section \ref{Sec: conclusion}.
	
	\section{Acceleration of degenerate proximal point methods}\label{Sec: acc dPPM}
	In this section, we start by introducing the dPPM. Subsequently, we present an accelerated version of the dPPM.
	
	\subsection{The degenerate proximal point method}
	Let $\cH$ be a real Hilbert space with inner product $\langle\cdot, \cdot\rangle$ and its induced norm $\|\cdot\|$. A set-valued operator $\cT: \cH \rightarrow 2^{\cH}$ is said to be a monotone operator if it satisfies the following inequality
	$$
	\left\langle v-v^{\prime}, w-w^{\prime}\right\rangle \geq 0 \quad \text { whenever } \quad v \in \cT w, v^{\prime} \in \cT w^{\prime} .
	$$
	It is said to be maximal monotone if, in addition, the graph
	$$
	{\rm{gph}}(\cT)=\{(w, v) \in \cH \times \cH \mid v \in \cT w\}
	$$
	is not properly contained in the graph of any other monotone operator $\cT^{\prime}: \cH\rightarrow 2^{\cH}$. Consider the following monotone inclusion problem:
	\begin{equation}\label{pro:inclusion-0}
		\text{ find } w\in \cH \text{ such that } 0\in \cT w,
	\end{equation}
	where $\cT$ is a maximal monotone operator from $\cH$ into itself. If one introduces a preconditioner, namely, a linear, bounded, self-adjoint, and positive semidefinite operator $\mathcal{M}: \mathcal{H} \rightarrow \mathcal{H}$, then the PPM with the preconditioner $\mathcal{M}$ \cite{bredies2022degenerate} for solving \eqref{pro:inclusion-0} can be expressed as follows:
	\begin{equation}\label{alg:dPPM0}
		w^0 \in \mathcal{H}, \quad w^{k+1}=  \hat{\mathcal{T}} w^k:= (\mathcal{M}+\mathcal{T})^{-1} \mathcal{M} w^k.
	\end{equation}
	Proper choices of $\cM$ will allow for efficient {evaluations} of $\hcT$. To ensure the well-definedness of \eqref{alg:dPPM0}, we introduce the concept of an admissible preconditioner:
	\begin{definition}[admissible preconditioner \cite{bredies2022degenerate}] An admissible preconditioner for the operator $\cT: \cH \rightarrow 2^{\cH}$ is a linear, bounded, self-adjoint, and positive semidefinite operator $\cM: \cH\rightarrow\cH$ such that
		\begin{equation}\label{def:hcT}
			\hcT=(\cM+\cT)^{-1}\cM
		\end{equation}
		is single-valued and has full domain.
	\end{definition}
	Drawing from classical results in functional operator theory, Bredies et al. \cite{bredies2022degenerate} {provided the following fundamental decomposition result to characterize linear, bounded, self-adjoint, and positive semidefinite operators}.
	\begin{proposition}[Proposition 2.3 in \cite{bredies2022degenerate}]\label{prop:M-decomposition}
		Let $\cM: \cH \rightarrow \cH$ be a linear, bounded, self-adjoint, and positive semidefinite operator. Then, there exists a bounded and injective operator $\cC: \cU \rightarrow  \cH$ for some real Hilbert space $\cU$, such that $\cM=\mathcal{C C}^*$, where $\cC^{*}:  \cH\to \cU$ is the adjoint of $\cC$. Moreover, if $\cM$ has closed range, then $\mathcal{C}^*$ is onto.
	\end{proposition}
	
	Let $\cM$ be an admissible preconditioner for the maximal monotone operator $\cT$. The dPPM for solving {the inclusion} problem \eqref{pro:inclusion-0} is expressed as follows:
	\begin{equation}\label{alg:dPPM}
		w^0 \in \cH, \quad w^{k+1}= (1-\rho_k) w^k + \rho_k \bw^k, {\quad \bw^k= {\hcT w^k= (\mathcal{M}+\mathcal{T})^{-1} \mathcal{M} w^k}, }
	\end{equation}
	where $\{\rho_k\}$ is a sequence in $[0,2]$. Here, $\cM$ is only required to be positive semidefinite, which is the reason we refer to it as a dPPM. Note that $\mathcal{M}$ is associated with a semi-inner-product $\langle v, w\rangle_{\mathcal{M}}:= \langle v, \mathcal{M} w\rangle$ for all $v$ and $w$ in $\mathcal{H}$, as well as a continuous seminorm $\|w\|_{\mathcal{M}}=$ $\sqrt{\langle w, w\rangle_{\mathcal{M}}}$. For notational convenience, define the following mappings:

	\begin{equation}\label{def:hcF}
		\hcQ:=\cI-\hcT \quad \text{ and }  \quad {\hcF_\rho:=(1-\rho)\cI+\rho \hcT, \quad \rho \in [0,2],}
	\end{equation}
	where $\cI$ is an identity operator on $\cH$. Clearly, if $0\in \cT w$, we have that $\hcT w=w$ and $\hcQ w=0$. Similar to \cite[Proposition 1]{rockafellar1976monotone}, we summarize some properties of $\hcT$, $\hcF_\rho$, and $\hcQ$ in the following proposition:
	\begin{proposition}\label{prop:M-nonexpansive} The following things hold:
		\begin{enumerate}
			\item[(a)]  $w=\hcT w+\hcQ w$ and $\cM\hcQ w\in \cT (\hcT w)$ for all $w\in \cH$;
			\item [(b)]  $\langle  \hcT w-\hcT w^{\prime}, \hcQ w-\hcQ w^{\prime} \rangle_{\cM} \geq 0 \text{ for all } w, w^{\prime} \in \cH; $
			\item [(c)] $\hcT$ is $\cM$-firmly nonexpansive, i.e.,
			$$\|\hcT w-\hcT w^{\prime}\|_{\cM}^2+\|\hcQ w-\hcQ w^{\prime}\|_{\cM}^2 \leq\|w-w^{\prime}\|_{\cM}^2, \quad \text{ for all } w, w^{\prime} \in \cH;$$
			\item [(d)] $\hcF_\rho$ is $\cM$-nonexpansive for $\rho\in (0,2]$, i.e.,
			$$\|\hcF_\rho w-\hcF_\rho w^{\prime}\|_{\cM}\leq\|w-w^{\prime}\|_{\cM}, \quad \text{ for all } w, w^{\prime} \in  \cH.$$
		\end{enumerate}
	\end{proposition}
	\begin{proof}
		The proofs of {parts} (a)-(c) can be obtained through straightforward calculations in a way similar to the proof provided in \cite[Proposition 1]{rockafellar1976monotone}. {Part (d)  follows from part (c), which can be verified by referring to \cite[Exercise 4.13]{bauschke2017convex}, with \(\|\cdot\|\) being replaced by \(\|\cdot\|_{\cM}\).} We omit the details here.
	\end{proof}
	 {Note that when \(\cM\) is the identity operator, \(\cM\)-firmly nonexpansiveness and \(\cM\)-nonexpansiveness reduce to the firmly nonexpansiveness and nonexpansiveness, respectively.} Furthermore, the global convergence of the dPPM in \eqref{alg:dPPM} can be established under the assumption that {$(\cM+\cT)^{-1}$ is $L$-Lipschitz; that is, there exists a constant $L \geq 0$ such that for all $v_1, v_2$ in the domain of $(\cM+\cT)^{-1}$,
		\[
		\|(\cM+\cT)^{-1}v_1 - (\cM+\cT)^{-1}v_2\| \leq L \|v_1 - v_2\|.
		\]
		Bredies et al. \cite{bredies2022degenerate} showed that this Lipschitz assumption holds for many splitting algorithms, including the Douglas-Rachford splitting method \cite{lions1979splitting} and the Chambolle-Pock scheme \cite{chambolle2011first}, under mild conditions. Additional examples can be found in Section 3 of \cite{bredies2022degenerate}. Moreover, building on the techniques developed in \cite{fazel2013hankel} for analyzing the global convergence of ADMM with semi-proximal terms, we provide a practical criterion in Section 3 for verifying this Lipschitz condition for the pADMM.}

	\begin{theorem}[Corollary 2.10 in \cite{bredies2022degenerate}]\label{Th: convergence-dPPM}
		Let $\cT:\cH \rightarrow 2^{\cH}$ with $\mathcal{T}^{-1}(0) \neq \emptyset$ be a maximal monotone operator and let $\cM$ be an admissible preconditioner such that $(\cM + \cT )^{-1}$ is L-Lipschitz. Let $\{w^k\}$ be any sequence generated by the dPPM in \eqref{alg:dPPM}. If $0<\inf_k \rho_k\leq \sup_k \rho_k<2$, then $\{w^k\}$ converges weakly to a point in $\mathcal{T}^{-1}(0)$.
	\end{theorem}
	Furthermore, Bredies et al. \cite{bredies2022degenerate} introduced a reduced preconditioned PPM for the parallel composition of $\cT$ by $\cC^*$, which can help us explore the acceleration of the dPPM in \eqref{alg:dPPM}. We summarize some key results related to parallel composition in the following proposition.
	
	\begin{proposition}\label{prop:tcT}
		Let $\cT:\cH \rightarrow 2^{\cH}$ be a maximal monotone operator and let $\cM$ be an admissible preconditioner with closed range. Suppose {that} $\cM = \cC \cC^{*}$ is a decomposition of $\cM$ according to Proposition \ref{prop:M-decomposition} with $\cC:\cU\rightarrow \cH$. The parallel composition $\mathcal{C}^* \triangleright \cT:=\left(\mathcal{C}^* \mathcal{T}^{-1} \mathcal{C}\right)^{-1}$ is a maximal monotone operator. Furthermore, the resolvent of $\mathcal{C}^* \triangleright \cT$ denoted by $\widetilde{\cT}:=\left(\cI+ \mathcal{C}^* \triangleright \cT\right)^{-1}$, has the following identity
		\begin{equation}\label{def:tcT}
			\widetilde{\cT}  =\mathcal{C}^*(\mathcal{M}+\cT)^{-1} \mathcal{C}.
		\end{equation}
		In particular, $\widetilde{\cT}:\cU\rightarrow\cU$ is everywhere well-defined and firmly nonexpansive. Moreover, for {any $\rho \in(0,2]$, $\tcF_\rho=(1-\rho)\cI+\rho \tcT$} is nonexpansive   and
		$$\mathcal{C}^*\cT^{-1}(0)=\mathcal{C}^* \operatorname{Fix} \hcT=\operatorname{Fix} \tcT=\operatorname{Fix} \tcF_\rho,
		$$
		where we denote the set of fixed-points of an operator $\hcT$ by $\operatorname{Fix} \hcT$.
	\end{proposition}
	\begin{proof}
		See Lemma 2.12 and Theorem 2.13 in \cite{bredies2022degenerate} for the maximal monotonicity of $\mathcal{C}^* \triangleright \cT$, the identity \eqref{def:tcT}, and the firm nonexpansiveness of $\tcT$. Moreover, it follows from \cite[Exercise 4.13]{bauschke2017convex} that   {for any  $\rho \in (0,2]$,  $\tcF_\rho$  is nonexpansive}.  Finally, from the proof of Theorem 2.14 in \cite{bredies2022degenerate}, we obtain $\mathcal{C}^* \operatorname{Fix} \hcT=\operatorname{Fix} \tcT$. Combining this with the fact $\cT^{-1}(0)=\operatorname{Fix} \hcT$, we conclude the proof.
	\end{proof}
	
	\subsection{An accelerated degenerate proximal point method}
	In this subsection, we consider the acceleration of the dPPM with a fixed relaxation parameter $\rho\in (0,2]$. Specifically, the dPPM in \eqref{alg:dPPM} can be reformulated as
	\begin{equation}\label{alg:dPPM-F}
		w^0 \in \cH, \quad w^{k+1}=\hcF_\rho w^k,
	\end{equation}
	where $\hcF_\rho$ defined in \eqref{def:hcF} is $\cM$-nonexpansive for $\rho\in(0,2]$ according to Proposition \ref{prop:M-nonexpansive}. Leveraging the $\mathcal{M}$-nonexpansiveness of $\hcF_\rho$, we propose the following accelerated dPPM presented in Algorithm \ref{alg:acc-dPPM}.
	\begin{algorithm}[htb]
		\caption{An accelerated dPPM for solving the inclusion problem \eqref{pro:inclusion-0}}
		\label{alg:acc-dPPM}
		\begin{algorithmic}[1]
			\STATE {Input: Let $\hat{w}^{0}=w^{0} \in \cH$, $\alpha\geq 2$ and $\rho\in(0,2]$. For $k=0,1, \ldots,$ perform the following steps in each iteration. }
			\STATE {Step 1. $\bw^{k}=\hcT w^{k}$.}
			\STATE {Step 2. $\hw^{k+1}=\hcF_\rho w^k= (1-\rho){w}^{k}+ \rho \bw^{k}$.}
			\STATE {Step 3. $w^{k+1}=w^k +\frac{\alpha}{2(k+\alpha)} (\hat{w}^{k+1}-w^k)+\frac{k}{k+\alpha}\left(\hat{w}^{k+1}-\hat{w}^{k}\right)$.}
		\end{algorithmic}
	\end{algorithm}
	\begin{remark}\label{Remark:connection}
		The proposed accelerated dPPM can be regarded as applying either Halpern's iteration \cite{lieder2021convergence} or the fast KM iteration \cite{bot2023fast} to the dPPM in \eqref{alg:dPPM-F}, depending on the choice of $\alpha$. On the one hand, when $\alpha=2$, Step 3 in Algorithm \ref{alg:acc-dPPM} {at the $k$-th interation} is
		\[
		w^{k+1} =  w^k + \frac{1}{k+2} (\hat{w}^{k+1}-w^{k}) + \frac{k}{k+2} \left(\hat{w}^{k+1} - \hat{w}^{k}\right).
		\]
		{Multiplying both sides of the above relation by $(k+2)$} {and rearranging it}, we have
		\begin{equation}\label{equ:s}
			(k+2)w^{k+1} - (k+1)\hat{w}^{k+1} = (k+1)w^k - k\hat{w}^{k}.
		\end{equation}
		Define the sequence $\{s^k\}$ as follows:
		$$s^{k} := (k+1)w^{k} - k \hat{w}^{k},\quad \forall k\geq 0.$$
		From \eqref{equ:s}, we can obtain $s^{k+1} = s^0$ for all $k\geq 0$. {This  implies}
		\begin{equation}\label{alg:w-halpern}
			w^{k+1} = \frac{1}{k+{2}}w^0 + \frac{k+1}{k+2}\hat{w}^{k+1},\quad \forall k \geq 0,
		\end{equation}
		{which}  is exactly the Halpern iteration \cite{halpern1967fixed,lieder2021convergence} applied to the dPPM in \eqref{alg:dPPM-F}. On the other hand, when $\alpha>2$, the proposed accelerated dPPM is equivalent to the fast KM iteration applied to \eqref{alg:dPPM-F} with a unit step size, as described in formula (12) of \cite{bot2023fast}.
	\end{remark}
	
	Let $\cM = \cC \cC^{*}$ be a decomposition of $\cM$ according to Proposition \ref{prop:M-decomposition} with $\cC:\cU\rightarrow \cH$. To discuss the global convergence and convergence rate of the accelerated dPPM in Algorithm \ref{alg:acc-dPPM}, we introduce two shadow sequences $\{u^k\}$ and $\{\bu^k\}$ defined as follows:
	\begin{equation}\label{alg:u}
		u^{k}:=\cC^{*}w^{k} \text{ and } \bu^k:=\cC^{*}\bw^{k},\quad \forall k\geq 0,
	\end{equation}
	where the sequences $\{w^k\}$ and $\{\bw^k\}$ are generated by Algorithm \ref{alg:acc-dPPM}. This leads to the following identity:
	\begin{equation}\label{alg:u-acc}
		u^{k+1}= u^k+\frac{\alpha}{2(k+\alpha)} (\tcF_\rho u^{k}-u^k)+\frac{k}{k+\alpha}\left(\tcF_\rho u^{k}-\tcF_\rho u^{k-1}\right),\quad \forall k\geq 1,
	\end{equation}
	{where $\tcF_\rho=(1-\rho)\cI+\rho \tcT$.} 
	Based on the shadow sequences in \eqref{alg:u} and the equivalence outlined in Remark \ref{Remark:connection}, we establish the global convergence of the proposed accelerated dPPM in the following theorem.	
	\begin{theorem}\label{Th:convergence-Halpern-dPPM}
		Let $\cT:\cH \rightarrow 2^{\cH}$ with $\mathcal{T}^{-1}(0) \neq \emptyset$ be a maximal monotone operator, and let $\cM$ be an admissible preconditioner with a closed range such that $(\cM + \cT )^{-1}$ is continuous. Suppose {that} $\cM = \cC \cC^{*}$ is a decomposition of $\cM$ according to Proposition \ref{prop:M-decomposition} with $\cC:\cU\rightarrow \cH$. The following conclusions hold for the sequences $\{\bar{w}^k\}$, $\{\hat{w}^{k}\}$, and $\{w^{k}\}$ generated by the accelerated dPPM in Algorithm \ref{alg:acc-dPPM}:
		\begin{enumerate}
			\item[(a)] If $\alpha=2$, then the sequence $\{\bar{w}^k\}$ converges strongly to a fixed  point  $w^{*} =(\mathcal{M}+\cT)^{-1} \mathcal{C}\Pi_{\cC^{*}\cT^{-1}(0)}(\cC^{*}w^0)$ in ${\cT}^{-1}(0)$, where $\Pi_{\cC^{*}\cT^{-1}(0)}(\cdot)$ is the projection operator onto the closed convex set $\cC^{*}\cT^{-1}(0)$; {moreover,} if  $\rho\in(0,2)$, then the sequences $\{w^{k}\}$ and $\{\hat{w}^{k}\}$ also converge strongly to $w^{*}$;
			\item[(b)] If $\alpha>2$, then the sequence $\{\bar{w}^k\}$ converges weakly to a fixed-point in ${\cT}^{-1}(0)$.
		\end{enumerate}
	\end{theorem}
	
	\begin{proof}
		We first establish the statement in part (a), where $\alpha=2$. According to the connection between the accelerated dPPM and the Halpern iteration outlined in \eqref{alg:w-halpern}, and the relationship between $\{u^k\}$ and $\{w^k\}$ in \eqref{alg:u}, we can obtain that
		\begin{equation}\label{Th:convergence-Halpern}
			u^{k+1}=\frac{1}{k+2} u^{0}+\frac{k+1}{k+2} \tcF_\rho u^k,\quad \forall k\geq 0.
		\end{equation}
		Note that $\tcF_\rho$ is nonexpansive for any $\rho \in(0,2]$  by Proposition \ref{prop:tcT}. It follows from the global convergence of the Halpern iteration in \cite[Theorem 2]{wittmann1992approximation} that
		\begin{equation}\label{Th:convergence-u}
			u^{k}\rightarrow \Pi_{\operatorname{Fix} \tcT}(u^{0}).
		\end{equation}
		By utilizing Proposition \ref{prop:tcT}, we have
		$$
		\mathcal{C}^*\mathcal{T}^{-1}(0)=\mathcal{C}^* \operatorname{Fix} \widehat{\mathcal{T}}=\operatorname{Fix} \tcT.
		$$
		Let $w^{*}=(\mathcal{M}+\mathcal{T})^{-1} \mathcal{C}\Pi_{\mathcal{C}^*\mathcal{T}^{-1}(0)}(u^{0})$. Consequently, by the relationship between $\{u^k\}$ and $\{w^k\}$ in \eqref{alg:u}, and \eqref{Th:convergence-u}, we can obtain
		\begin{equation}\label{Th:convergence-barw}
			\bar{w}^k=(\mathcal{M}+\mathcal{T})^{-1}\mathcal{C}\mathcal{C}^{*}w^{k}=(\mathcal{M}+\mathcal{T})^{-1}\mathcal{C} u^{k}\rightarrow (\mathcal{M}+\mathcal{T})^{-1} \mathcal{C}\Pi_{\mathcal{C}^*\mathcal{T}^{-1}(0)}(u^{0})=w^{*},
		\end{equation}
		where the continuity of $(\mathcal{M}+\mathcal{T})^{-1}\mathcal{C}$ is derived from the composition of a continuous function $(\mathcal{M}+\mathcal{T})^{-1}$ and a linear operator $\mathcal{C}$. Hence, $\{\bar{w}^k\}$ converges strongly to $w^*$. Furthermore, since $\tcT \Pi_{\mathcal{C}^*\mathcal{T}^{-1}(0)}(u^{0})=\Pi_{\mathcal{C}^*\mathcal{T}^{-1}(0)}(u^{0})$, we have
		$$
		\begin{array}{ll}
			\hcT w^*&=(\mathcal{M}+\mathcal{T})^{-1} \cC\cC^*(\mathcal{M}+\mathcal{T})^{-1} \mathcal{C}\Pi_{\mathcal{C}^*\mathcal{T}^{-1}(0)}(u^{0})\\
			&= (\mathcal{M}+\mathcal{T})^{-1} \cC \tcT  \Pi_{\mathcal{C}^*\mathcal{T}^{-1}(0)}(u^{0})  \\
			&=(\mathcal{M}+\mathcal{T})^{-1} \cC \Pi_{\mathcal{C}^*\mathcal{T}^{-1}(0)}(u^{0}) =w^*,
		\end{array}
		$$
		which implies that $w^{*} \in \operatorname{Fix} \widehat{\mathcal{T}}=\mathcal{T}^{-1}(0)$. It remains to show that $\{w^{k}\}$ and $\{\hat{w}^{k}\}$ also converge strongly to $w^{*}$ for $\rho \in (0,2)$. According to \eqref{alg:w-halpern}, we have
		\begin{equation}\label{Th:convergence-w}
			\begin{array}{ll}
				&\|w^{k+1}-w^{*}\| \\
				=& \displaystyle \|\frac{1}{k+2} (w^{0}-w^{*})+ \frac{k+1}{k+2}((1-\rho)(w^{k}-w^{*})+\rho (\bar{w}^k-w^{*}))\|  \\
				\displaystyle 	\leq& \displaystyle\frac{k+1}{k+2}|1-\rho|\|w^{k}-w^{*}\|+ \frac{1}{k+2} \| w^{0}-w^{*}\|+ \frac{k+1}{k+2}\rho \|\bar{w}^k-w^{*}\|.
			\end{array}
		\end{equation}
		For any $k\geq 0$, define $q_k:= \frac{k+1}{k+2}|1-\rho|$ and $\beta_k:=\frac{1}{k+2} \| w^{0}-w^{*}\|+ \frac{k+1}{k+2}\rho \|\bar{w}^k-w^{*}\|.$ {It is evident that} $0\leq q_k<1$, $\beta_k\geq 0$ for any $k\geq0$, and
		\begin{equation}\label{Th:convergence-q-1}
			\sum_{k=0}^{\infty}(1-q_k)= \sum_{k=0}^{\infty} 1- \left(\frac{k+1}{k+2}\right)|1-\rho|= \sum_{k=0}^{\infty} \frac{1}{k+2}+ \left(\frac{k+1}{k+2}\right)(1-|1-\rho|)=+\infty.
		\end{equation}
		Furthermore, due to the convergence of $\{\bar{w}^k\}$, we can deduce
		\begin{equation}\label{Th:convergence-q-2}
			\frac{\beta_k}{1-q_k}\rightarrow 0,\quad \forall \rho \in (0,2).
		\end{equation}
		It follows from \eqref{Th:convergence-q-1}, \eqref{Th:convergence-q-2}, and \cite[Lemma 3 in subsection 2.2.1]{polyak1987introduction} that
		\begin{equation*}
			\|w^{k+1}-w^{*}\|\rightarrow 0,
		\end{equation*}
		which in accordance with \eqref{alg:w-halpern} implies
		\begin{equation*}
			\|\hat{w}^{k}-w^{*}\|\rightarrow 0.
		\end{equation*}
		Hence, $\{w^{k}\}$ and $\{\hat{w}^{k}\}$ also converge strongly to $w^{*}$ for $\rho\in(0,2)$.
		\par
		Now, we shall prove part (b) where $\alpha>2$. According to the identity in \eqref{alg:u-acc}, the nonexpansiveness of $\tcF_\rho$, and the convergence result of the fast KM \cite[Theorem 3.4]{bot2023fast}, we can obtain that $\{u^k\}$ converges weakly to an element in $\operatorname{Fix} \tcT$. Similar to \eqref{Th:convergence-barw}, we can deduce that $\{\bar{w}^k\}$ converges weakly to a point in ${\cT}^{-1}(0)$.
	\end{proof}

	\begin{remark}
		The acceleration step (Step 3) in Algorithm \ref{alg:acc-dPPM} can be replaced by other acceleration techniques \cite{tran2023halpern}. Following a similar approach to the proof of Theorem \ref{Th:convergence-Halpern-dPPM}, we can obtain the global convergence of the sequence $\{\bw^k\}$ if the shadow sequence $\{u^k\}$ converges.
	\end{remark}
	
	The following proposition provides convergence rates for the accelerated dPPM.
	\begin{proposition}\label{prop:complexity-acc-dPPM}
		Let $\cT:\cH \rightarrow 2^{\cH}$ with $\mathcal{T}^{-1}(0) \neq \emptyset$ be a maximal monotone operator and let $\cM$ be an admissible preconditioner with closed range. The sequences $\{w^{k}\}$ and $\{\hat{w}^{k}\}$ generated by Algorithm \ref{alg:acc-dPPM} {satisfy  the following}:
		\begin{enumerate}
			\item[(a)] if $\alpha=2$, then
			\begin{equation}\label{prop:complexity-acc-dPPM-2}
				\|w^{k}-\hat{w}^{k+1}\|_{\cM} \leq \frac{2\left\|w^0-w^*\right\|_{\cM}}{k+1},\quad  \forall k \geq 0 \text { and } w^* \in \mathcal{T}^{-1}(0);
			\end{equation}
			\item[(b)] if $\alpha>2$, then
			\begin{equation}\label{prop:complexity-acc-dPPM>2}
				\|w^{k+1}-w^{k}\|_{\cM}=o\left(\frac{1}{k+1}\right)  \text { and }  \|w^{k}-\hat{w}^{k+1}\|_{\cM}=o\left(\frac{1}{k+1}\right) \text { as } k \rightarrow+\infty \text {.}
			\end{equation}
		\end{enumerate}
	\end{proposition}
	
	\begin{proof}
		Let $\cM = \cC \cC^{*}$ be a decomposition of $\cM$ according to Proposition \ref{prop:M-decomposition} with $\cC:\cU\rightarrow\cH$. If $\alpha=2$, then the shadow sequence $\{u^k\}$ satisfying \eqref{Th:convergence-Halpern} is exactly the Halpern iteration. By Proposition \ref{prop:tcT}, we know that $\tcF_\rho$ is nonexpansive for $\rho \in(0,2]$. It follows from \cite[Theorem 2.1]{lieder2021convergence} that
		\begin{equation}\label{prop:complexity-acc-dPPM-1}
			\|u^{k}-\tcF_\rho u^k\| \leq \frac{2\left\|u^0-u^*\right\|}{k+1}, \quad \forall k \geq 0 \text { and } u^* \in \operatorname{Fix}\tcF_\rho.
		\end{equation}
		Due to Proposition \ref{prop:tcT}, we also have $\operatorname{Fix}\tcF_\rho=\cC^{*}\cT^{-1}(0)$. Thus, for any $u^* \in \operatorname{Fix}\tcF_\rho$, there exists a point $w^{*}=(\mathcal{M}+\cT)^{-1} \mathcal{C}u^{*} \in \cT^{-1}(0) $ such that $\cC^{*}w^{*}=u^{*}$. Hence, \eqref{prop:complexity-acc-dPPM-1} can be rewritten as:
		\begin{equation*}
			\|\cC^{*}w^{k}-\cC^{*} \hcF_\rho w^k\| \leq \frac{2\left\|\cC^{*}w^{0}-\cC^{*}w^{*}\right\|}{k+1}, \quad \forall k \geq 0 \text { and } w^* \in \cT^{-1}(0),
		\end{equation*}
		which implies
		\begin{equation*}
			\|w^{k}-\hat{w}^{k+1}\|_{\cM} =\|w^{k}-\hcF_\rho w^k\|_{\cM} \leq \frac{2\left\|w^{0}-w^{*}\right\|_{\cM}}{k+1}, \quad \forall k \geq 0 \text { and } w^* \in \cT^{-1}(0).
		\end{equation*}
		
		If $\alpha>2$, then the shadow sequence $\{u^k\}$ satisfying \eqref{alg:u-acc} is exactly the fast KM iteration. It follows from \cite[Theorem 3.5]{bot2023fast} that
		\begin{equation*}
			\|u^{k+1}-u^{k}\|=o\left(\frac{1}{k+1}\right) \quad \text { and } \quad\|u^{k}-\tcF_\rho u^{k}\|=o\left(\frac{1}{k+1}\right) \text { as } k \rightarrow+\infty \text {. }
		\end{equation*}
		Similar to the previous case {with} $\alpha=2$, we can conclude that
		\begin{equation*}
			\|w^{k+1}-w^{k}\|_{\cM}=o\left(\frac{1}{k+1}\right) \quad \text { and } \quad\|w^{k}-\hat{w}^{k+1}\|_{\cM}=o\left(\frac{1}{k+1}\right) \text { as } k \rightarrow+\infty \text {, }
		\end{equation*}
		which completes the proof.
	\end{proof}
	\begin{remark}
		The convergence rates discussed in Proposition \ref{prop:complexity-acc-dPPM} are essentially inherited from the results of the Halpern iteration \cite{lieder2021convergence} and the fast KM iteration \cite{bot2023fast} applied to the nonexpansive operator $\tcF_\rho$ defined in Proposition \ref{prop:tcT}. Moreover, without acceleration, similar to \cite[Proposition 8]{brezis1978produits}, we can obtain that the dPPM with $\rho=1$ has an $O(1/\sqrt{k})$ convergence rate with respect to $\|w^k-w^{k-1}\|_{\cM},\ k\geq 1$.
	\end{remark}
	The following corollary demonstrates that even when the proximal term $\cM$ is positive semidefinite, the accelerated dPPM can achieve an $O(1/k)$ convergence rate in terms of the operator residual under the norm.
	\begin{corollary}\label{coro:complexity-HPPPA}
		Let $\cT:\cH \rightarrow 2^{\cH}$ with $\mathcal{T}^{-1}(0) \neq \emptyset$ be a maximal monotone operator and let $\cM$ be an admissible preconditioner with closed range such that $(\cM + \cT )^{-1}$ is L-Lipschitz. Suppose {that} $\cM = \cC \cC^{*}$ is a decomposition of $\cM$ according to Proposition \ref{prop:M-decomposition} with $\cC:\cU\rightarrow\cH$.   {Let  $\|\cC\| := \sup_{\|w\| \leq 1}\|\cC w\|$ represent  the norm of the linear operator $\cC$.}  Choose $\alpha=2$ and $\rho=1$.  Then the sequences $\{w^{k}\}$ and $\{\bar{w}^{k}\}$ generated by Algorithm \ref{alg:acc-dPPM} {satisfiy}
		$$
		\|w^{k}-\bar{w}^{k}\|\leq \frac{1}{k+1}\|w^{0}-w^*\| +\frac{(5k+1)L\|\mathcal{C}\|}{(k+1)^2}  \|w^{0}-w^{*}\|_{\cM}
		,\ \forall k\geq 0 \text{ and } w^* \in \mathcal{T}^{-1}(0).
		$$
	\end{corollary}
	\begin{proof}
		Set $\alpha=2$ and $\rho=1$. Then according to \eqref{alg:w-halpern}, we have
		\begin{equation}\label{complexity-HPPPA-1}
			w^{k+1}= \frac{1}{k+2}w^{0}+\frac{k+1}{k+2}\hcT w^{k}, \quad \forall k\geq 0.
		\end{equation}
		Since $(\mathcal{M}+\cT)^{-1}$ is $L$-Lipschitz, we have for all $w^{\prime}, w^{\prime \prime} \in \cH$,
		\begin{equation*}\label{complexity-HPPPA-2}
			\|\hcT w^{\prime}-\hcT w^{\prime \prime}\|=\left\|(\mathcal{M}+\mathcal{T})^{-1} \mathcal{C C}^* w^{\prime}-(\mathcal{M}+\mathcal{T})^{-1} \mathcal{C} \mathcal{C}^* w^{\prime \prime}\right\| \leq L\|\mathcal{C}\|\left\|w^{\prime}-w^{\prime \prime}\right\|_{\mathcal{M}},
		\end{equation*}
		which,  {together with \eqref{complexity-HPPPA-1}},  yields that for any $w^{*}\in \mathcal{T}^{-1}(0)$ and $k\geq 0$,
		\begin{equation}\label{complexity-HPPPA-3}
			\begin{array}{ll}
				&\|w^{k+1}-\hcT w^{k+1}\|\\
				=& \|\frac{1}{k+2}[(w^{0}-w^{*})+ (w^{*}-\hcT w^{k+1})]+\frac{k+1}{k+2}(\hcT w^{k}-\hcT w^{k+1})\|\\
				\leq& \frac{1}{k+2}\left(\|w^{0}-w^*\| + \|w^{*}-\hcT w^{k+1}\|\right)+\frac{k+1}{k+2}\|\hcT w^{k}-\hcT w^{k+1}\|\\
				\leq& \frac{1}{k+2}\left(\|w^{0}-w^*\| + L\|\mathcal{C}\|\|w^{*}-w^{k+1}\|_{\cM}\right)+ \frac{k+1}{k+2}L\|\mathcal{C}\|\left\|w^{k}- w^{k+1}\right\|_{\cM}.
			\end{array}
		\end{equation}
		Next, we shall estimate $\|w^{k}- w^{k+1}\|_{\cM}$. Indeed, from \eqref{complexity-HPPPA-1}, we have
		\begin{equation}\label{complexity-HPPPA-4}
			\begin{array}{ll}
				&\|w^{k}- w^{k+1}\|_{\cM}\\
				=& \|\frac{1}{k+2}[(w^{0}-w^{*})+ (w^{*}-w^{k})] +\frac{k+1}{k+2}(\hcT w^{k}-  w^{k})\|_{\cM}\\
				\leq&  \frac{1}{k+2}\left(\|w^{0}-w^{*}\|_{\cM}+ \|w^{*}-w^{k}\|_{\cM}\right)  +\frac{k+1}{k+2}\|\hcT w^{k}-  w^{k}\|_{\cM}.\\
			\end{array}
		\end{equation}

		Now, we claim that $\| w^{k}  -w^{*}\|_{\mathcal{M}}\leq \| w^{0}  -w^{*}\|_{\mathcal{M}}$ for any $k\geq 0$. Assume $\| w^{k}  -w^{*}\|_{\mathcal{M}}\leq \| w^{0}  -w^{*}\|_{\mathcal{M}}$ for some $k\geq 0$ ({this is true for $k=0$ trivially}). Since $\hcT$ is $\cM$-nonexpansive by Proposition \ref{prop:M-nonexpansive}, we have
		\begin{equation*}
			\begin{aligned}
				\| w^{k+1}  -w^{*}\|_{\mathcal{M}} &=\|\frac{1}{k+2}\left(w^0-w^{*}\right)+\frac{k+1}{k+2} (\hcT w^k-w^{*}) \|_{\mathcal{M}} \\
				&\leq \frac{1}{k+2}\|w^{0}-w^{*}\|_{\cM}+\frac{k+1}{k+2}\| w^{k}-w^{*}\|_{\cM}\\
				&\leq \| w^{0}-w^{*}\|_{\cM}.
			\end{aligned}
		\end{equation*}
		{In particular, this  implies that $\| w^{1}  -w^{*}\|_{\mathcal{M}}\leq \|w^{0}-w^{*}\|_{\cM}$.}
		{Thus, we derive from the  induction that}
		\begin{equation}\label{complexity-HPPPA-5}
			\| w^{k}  -w^{*}\|_{\mathcal{M}}\leq \|w^{0}-w^{*}\|_{\cM},\quad \forall k \geq 0.
		\end{equation}
		Thus, from \eqref{complexity-HPPPA-3}, \eqref{complexity-HPPPA-4}, and \eqref{complexity-HPPPA-5}, we can deduce that for any $k\geq 0$,
		\begin{equation}\label{complexity-HPPPA-6}
			\begin{array}{ll}
				\|w^{k+1}-\hcT w^{k+1}\|\leq &\frac{1}{k+2}\left(\|w^{0}-w^*\| + L\|\mathcal{C}\|\|w^{*}-w^{0}\|_{\cM}\right) \\
				&\quad +\frac{k+1}{k+2}L\|\mathcal{C}\|\left(\frac{2}{k+2}\|w^{0}-w^{*}\|_{\cM}  +\frac{k+1}{k+2}\|\hcT w^{k}-  w^{k}\|_{\cM}\right).
			\end{array}
		\end{equation}
		Also, from Proposition \ref{prop:complexity-acc-dPPM}, we have for $\rho=1$,
		\begin{equation*}
			\|w^{k}-\hcT w^k\|_{\cM} \leq \frac{2\left\|w^0-w^*\right\|_{\cM}}{k+1},\ \forall k \geq 0 \text { and } w^* \in \mathcal{T}^{-1}(0),
		\end{equation*}
		which together with  \eqref{complexity-HPPPA-6} yields that for any $k\geq 0$,
		\begin{equation*}
			\begin{array}{ll}
				\|w^{k+1}-\hcT w^{k+1}\|&\leq  \frac{1}{k+2}\|w^{0}-w^*\|  +\frac{5k+6}{(k+2)^2}L\|\mathcal{C}\|\|w^{0}-w^{*}\|_{\cM}.
			\end{array}
		\end{equation*}
		Note that for $k=0$, we also have
		\begin{equation*}
			\begin{array}{ll}
				\|w^{0}-\hcT w^{0}\|&\leq \|w^{0}-w^*\|  +\| \hcT w^{0}-w^{*} \|\\
				&\leq \|w^{0}-w^*\|  + L\|\mathcal{C}\|\|w^{0}-w^{*}\|_{\cM}.
			\end{array}
		\end{equation*}
		Hence, for any $k\geq 0$,
		\begin{equation*}
			\|w^{k}-\hcT w^{k}\|\leq \frac{1}{k+1}\|w^{0}-w^*\|  +\frac{5k+1}{(k+1)^2}L\|\mathcal{C}\|\|w^{0}-w^{*}\|_{\cM},
		\end{equation*}
		which completes the proof.
	\end{proof}
	
	\section{Acceleration of the pADMM} \label{Sec: acc-pADMM}
	In this section, we first establish the equivalence between the pADMM and the dPPM. Subsequently, we introduce an accelerated pADMM based on the accelerated dPPM. Finally, we present the global convergence of the accelerated pADMM and discuss its convergence rate.
	
	\subsection{The equivalence between the pADMM and the dPPM}
	We first introduce some quantities for further analysis. Specifically, we consider the following self-adjoint linear operator $\cM: \mathbb{W} \rightarrow \mathbb{W} $ defined by
	\begin{equation}\label{def:M}
		\cM=\left[\begin{array}{ccc}
			\sigma B_{1}^{*}B_{1}+\cT_1 & 0 & B_{1}^{*} \\
			0 &\cT_2 & 0 \\
			B_{1} & 0 & \sigma^{-1} \cI
		\end{array}\right].
	\end{equation}
	Furthermore, since $f_{1}$ and $f_{2}$ are proper closed convex functions, there exist two self-adjoint and positive semidefinite operators $\Sigma_{f_{1}}$ and $\Sigma_{f_{2}}$ such that for all $y, \hat{y} \in \operatorname{dom}(f_{1}), \phi \in \partial f_{1}(y)$, and $\hat{\phi} \in \partial f_{1}(\hat{y})$,
	$$f_{1}(y) \geq f_{1}(\hat{y})+\langle\hat{\phi}, y-\hat{y}\rangle+\frac{1}{2}\|y-\hat{y}\|_{\Sigma_{f_{1}}}^{2} \text{ and } \langle \phi-\hat{\phi}, y-\hat{y}\rangle \geq\|y-\hat{y}\|_{\Sigma_{f_{1}}}^{2},$$
	and for all $z, \hat{z} \in \operatorname{dom}(f_{2}), \varphi \in \partial f_{2}(z)$, and $\hat{\varphi} \in \partial f_{2}(\hat{z})$,
	$$\quad f_{2}(z) \geq f_{2}(\hat{z})+\langle\hat{\varphi}, z-\hat{z}\rangle+\frac{1}{2}\|z-\hat{z}\|_{\Sigma_{f_{2}}}^{2} \text{ and } \langle \varphi-\hat{\varphi}, z-\hat{z}\rangle \geq\|z-\hat{z}\|_{\Sigma_{f_{2}}}^{2}.$$
	On the other hand, it follows from \cite[Corollary 28.3.1]{rockafellar1970convex} that $({y}^*, {z}^*) \in$ $\mathbb{Y} \times \mathbb{Z}$ is an optimal solution to problem \eqref{primal} if and only if there exists ${x}^* \in \mathbb{X}$ such that $({y}^*, {z}^*,{x}^*)$ satisfies the following KKT system:
	\begin{equation}\label{eq:KKT}
		\quad -B_1^* {x}^* \in \partial f_{1}({y}^*), \quad -B_{2}^* {x}^* \in \partial f_2({z}^*), \quad B_{1} {y}^*+B_{2} {z}^*-c=0,
	\end{equation}
	where $\partial f_1$ and $\partial f_2$ are the subdifferential mappings of $f_1$ and $f_2$, respectively. Then, solving problem \eqref{primal} is equivalent to finding $w\in \mathbb{W}$ such that $0\in \cT w$, where the maximal monotone operator $\cT$ is defined by
	\begin{equation}\label{def:T}
		\cT w=\left(\begin{array}{c}
			\partial f_{1}({y}) +B_1^* {x}  \\
			\partial f_{2}({z}) +B_2^* {x}  \\
			c-B_1{y}-B_2{z} \\
		\end{array}     \right),  \quad  \forall w=(y,z,x)\in \mathbb{W}:=.
	\end{equation}
	Now, we make the following assumptions:
	\begin{assumption}\label{ass: CQ}
		The KKT system \eqref{eq:KKT} has a nonempty solution set.
	\end{assumption}
	\begin{assumption}
		\label{ass: Assump-solvability}
		{Both $\Sigma_{f_1} +   B_1^*B_1+\cT_1$ and $\Sigma_{f_2} +  B_2^*B_2+\cT_2$ are positive definite.}
	\end{assumption}
	Under Assumption \ref{ass: Assump-solvability} originated from \cite{fazel2013hankel}, each step of the pADMM is well-defined due to the strong convexity of the objective functions in the subproblems. Assumption \ref{ass: Assump-solvability} holds automatically if either $B_i$ is injective or $f_i$ is strongly convex (for $\left.i=1,2\right)$. Han et al. \cite{han2018linear} presented an example where neither $B_i$ is injective nor $f_i$ is strongly convex (for $i=1,2$), yet Assumption \ref{ass: Assump-solvability} can still hold {easily}. Furthermore, the following lemma shows that $(\cM + \cT )^{-1}$ is Lipschitz continuous under Assumption \ref{ass: Assump-solvability}.
	\begin{lemma}\label{lem:inveritble}
		Suppose that Assumption \ref{ass: Assump-solvability} holds. Consider {the operators} $\cT$ defined in \eqref{def:T} and $\cM$ defined in \eqref{def:M}. Then, $(\cM + \cT )^{-1}$ is Lipschitz continuous.
	\end{lemma}
	\begin{proof}
		We begin by establishing that $(\cM+\cT)^{-1}$ is single-valued    {by proof of contradiction.} Assume, for the sake of contradiction, that $(\cM+\cT)^{-1}$ is not single-valued. Then, there exist distinct $\bw_1=(\by_1,\bz_1,\bx_1)\in \mathbb{W}$ and $\bw_2=(\by_2,\bz_2,\bx_2)\in \mathbb{W}$ such that $\bw_1\in (\cM+\cT)^{-1}v$ and $\bw_2\in (\cM+\cT)^{-1}v$ for some $v=(v_y,v_z,v_x)\in \mathbb{W}$, which implies, for $i=1,2$,
		\begin{eqnarray}
			&&v_y\in \partial f_1(\by_i)+(\sigma B_{1}^{*}B_{1}+\cT_1)\by_i+2B_{1}^{*}\bx_{i},\label{lem:inveritble-1}\\
			&&v_z\in \partial f_2(\bz_i)+B_{2}^{*}\bx_{i}+\cT_{2}\bz_{i},\label{lem:inveritble-2}\\
			&&v_x = c-B_{2}\bz_i +\sigma^{-1}\bx_i.\label{lem:inveritble-3}
		\end{eqnarray}
		Using the definitions of $\Sigma_{f_{1}}$ and $\Sigma_{f_{2}}$, we derive the following inequalities from equations \eqref{lem:inveritble-1}, \eqref{lem:inveritble-2}, and \eqref{lem:inveritble-3}:
		\begin{eqnarray}
			&&\langle -(\sigma B_{1}^{*}B_{1}+\cT_1)(\by_1-\by_2)-2B_{1}^{*}(\bx_1-\bx_2)   , \by_1-\by_2 \rangle\geq \|\by_1-\by_2\|_{\Sigma_{f_1}}^2,\label{lem:inveritble-4}\\
			&&\langle -\cT_2(\bz_1-\bz_2)-B_{2}^{*}(\bx_1-\bx_2), \bz_1-\bz_2 \rangle\geq \|\bz_1-\bz_2\|_{\Sigma_{f_2}}^2,\label{lem:inveritble-5}\\
			&&\sigma B_{2}(\bz_1-\bz_2)=\bx_1-\bx_2.\label{lem:inveritble-6}
		\end{eqnarray}
		Substituting \eqref{lem:inveritble-6} into \eqref{lem:inveritble-5}, we have
		$$
		0\geq \|\bz_1-\bz_2\|_{\Sigma_{f_2}+\cT_2+\sigma B_2^{*}B_2}^2.
		$$
		As $\Sigma_{f_2}+\cT_2+\sigma B_2^{*}B_2$ is positive definite by Assumption \ref{ass: Assump-solvability}, we conclude that
		\begin{equation}\label{lem:inveritble-7}
			\bz_1-\bz_2=0,
		\end{equation}
		which together with \eqref{lem:inveritble-6} yields that
		\begin{equation}\label{lem:inveritble-8}
			\bx_1-\bx_2=0.
		\end{equation}
		Similarly, substituting \eqref{lem:inveritble-8} into \eqref{lem:inveritble-4}, we can deduce from the positive definiteness of $\Sigma_{f_1}+\cT_1+\sigma B_1^{*}B_1$ by Assumption \ref{ass: Assump-solvability} that
		\begin{equation}\label{lem:inveritble-9}
			\by_1-\by_2=0.
		\end{equation}
		Therefore, equations \eqref{lem:inveritble-7}, \eqref{lem:inveritble-8}, and \eqref{lem:inveritble-9} show that $\bw_1=\bw_2$, contradicting our assumption. Thus, we conclude that $(\cM+\cT)^{-1}$ is indeed single-valued.

		\par To show the Lipschitz continuity of $(\cM+\cT)^{-1}$,
		consider $\bw_i=(\by_i,\bz_i,\bx_i)\in \mathbb{W}$ such that $\bw_i= (\cM+\cT)^{-1}v_i$, where $v_i=(v_{iy},v_{iz},v_{ix})$ for $i=1,2$. Similar to \eqref{lem:inveritble-1}, \eqref{lem:inveritble-2}, and \eqref{lem:inveritble-3}, we have for $i=1,2,$
		\begin{eqnarray}
			&&v_{iy}-2B_{1}^{*}\bx_{i}\in (\partial f_1+\sigma B_{1}^{*}B_{1}+\cT_1)\by_i,\\
			&&v_{iz}-\sigma B_{2}^{*}(v_{ix}-c)\in (\partial f_2+\sigma B_2^{*}B_2+\cT_2)\bz _i ,\\
			&&v_{ix} = c-B_{2}\bz_i +\sigma^{-1}\bx_i   \label{lem:inveritble-10}.
		\end{eqnarray}
		By \cite[Corollary 23.5.1]{rockafellar1970convex}, we can obtain for $i=1,2,$
		\begin{eqnarray*}
			&&\by_i\in (\partial f_1+\sigma B_{1}^{*}B_{1}+\cT_1)^{-1}(v_{iy}-2B_{1}^{*}\bx_{i}),\\
			&&\bz _i\in (\partial f_2+\sigma B_2^{*}B_2+\cT_2)^{-1}(v_{iz}-\sigma B_{2}^{*}(v_{ix}-c)).
		\end{eqnarray*}
		Since $\partial f_2+\sigma B_{2}^{*}B_{2}+\cT_2$ is strongly monotone by Assumption \ref{ass: Assump-solvability},  $(\partial f_2+\sigma B_{2}^{*}B_{2}+\cT_2)^{-1}$ is Lipschitz continuous by \cite[Proposition 12.54]{rockafellar1998variational}. Therefore, there exists a constant $L_2$ such  that
		\begin{eqnarray}\label{lem:inveritble-11}
			\begin{array}{ll}
				\|\bz _1-\bz_2\|&\leq L_2\|(v_{1z}-v_{2z})-\sigma B_{2}^{*}(v_{1x}-v_{2x})\|\\
				&\leq L_2\|v_{1z}-v_{2z}\|+\sigma L_2\|B_{2}^{*}\|\| v_{1x}-v_{2x}\|.
			\end{array}
		\end{eqnarray}
		Hence, by \eqref{lem:inveritble-10} and \eqref{lem:inveritble-11}, we can obtain
		\begin{equation}\label{lem:inveritble-12}
			\begin{array}{ll}
				\| \bx_1-\bx_2\|&=\|\sigma (v_{1x}-v_{2x}+B_{2}(\bz_1-\bz_2))\| \\
				&\leq \sigma(1+\sigma L_2\|B_2\|\|B_2^{*}\|)\|v_{1x}-v_{2x}\|+\sigma L_2\|B_2\|\|v_{1z}-v_{2z}\|.
			\end{array}
		\end{equation}
		Similarly, since $\partial f_1+\sigma B_{1}^{*}B_{1}+\cT_1$ is strongly monotone by Assumption \ref{ass: Assump-solvability}, there exists a constant $L_1$ such  that
		\begin{equation}\label{lem:inveritble-13}
			\begin{array}{ll}
				\|\by_1-\by_2\|&\leq L_1\|(v_{1y}-v_{2y})-2B_{1}^{*}(\bx_{1}-\bx_2)\|\\
				&\leq  L_1(\|v_{1y}-v_{2y}\|+2\|B_{1}^{*}\|\|\bx_{1}-\bx_2\|)\\
				&\leq L_1(\|v_{1y}-v_{2y}\|+2\|B_{1}^{*}\|\sigma(1+\sigma L_2\|B_2\|\|B_2^{*}\|)\|v_{1x}-v_{2x}\| \\
				& \quad +2\|B_{1}^{*}\|\sigma L_2\|B_2\|\|v_{1z}-v_{2z}\|).
			\end{array}
		\end{equation}
		Therefore, by \eqref{lem:inveritble-11}, \eqref{lem:inveritble-12}, and \eqref{lem:inveritble-13}, there exists a constant $L$ such that $\|\bw_1-\bw_2\|\leq L\|v_1-v_2\|$.
	\end{proof}
	
	{Inspired by the interpretation of pADMM as a (partial) PPM in \cite{chen2015inertial} (for $\rho=1$) and \cite{yang2023accelerated} (for $\rho \in \mathbb{R}$), below we rigorously establish an important one-to-one correspondence between pADMM and dPPM.
	}
	\begin{proposition}\label{prop:equ-PADMM-dPPM}
		Suppose that Assumption \ref{ass: Assump-solvability} holds. Consider {the operators} $\cT$ defined in \eqref{def:T} and $\cM$ defined in \eqref{def:M}. Then the sequence $\left\{w^k\right\}$ generated by the pADMM in Algorithm \ref{alg:pADMM} coincides with the sequence $\left\{w^k\right\}$ generated by the dPPM in \eqref{alg:dPPM} with the same initial point $ w^0 \in \mathbb{W}$. Additionally, $\cM$ is an admissible preconditioner such that $(\cM + \cT )^{-1}$ is Lipschitz continuous.
	\end{proposition}
	\begin{proof}
		Similar to the proof establishing the equivalence of the pADMM and the (partial) PPM, as outlined in Appendix B of \cite{yang2023accelerated}, we can obtain that under Assumption \ref{ass: Assump-solvability}, the sequence $\left\{w^k\right\}$ generated by Algorithm \ref{alg:pADMM} coincides with the sequence $\left\{w^k\right\}$ generated by the following scheme:
		\begin{equation*}
			\begin{array}{l}
				\quad \cM w^k  \in (\cM+\cT) \bw^k, \quad w^{k+1}= (1-\rho_k) w^k + \rho_k \bw^k.
			\end{array}
		\end{equation*}
		Since $(\cM+\cT)^{-1}$ is single-valued according to Lemma \ref{lem:inveritble}, we have
		\begin{equation*}
			\begin{array}{l}
				\quad \bw^k=(\cM+\cT)^{-1}\cM w^k, \quad w^{k+1}= (1-\rho_k) w^k + \rho_k \bw^k.
			\end{array}
		\end{equation*}
		Note that for {an arbitrary choice} of $w^k\in \mathbb{W}$, each step in Algorithm \ref{alg:pADMM} is well-defined under Assumption \ref{ass: Assump-solvability}. Hence, based on the equivalence established, we conclude that $(\cM+\cT)^{-1}\cM$ has full domain. Combining this result with Lemma \ref{lem:inveritble}, we deduce that $\cM$ is an admissible preconditioner such that $(\cM + \cT )^{-1}$ is Lipschitz continuous, which completes the proof.
	\end{proof}
	
	The convergence of Algorithm \ref{alg:pADMM} with the uniform relaxation factor $\rho \in (0,2)$ has already been established in \cite{xiao2018generalized} under Assumptions \ref{ass: CQ} and \ref{ass: Assump-solvability} {by borrowing the proofs from \cite{fazel2013hankel}.}
	In contrast, by leveraging the equivalence between the dPPM and the pADMM, as detailed in Proposition \ref{prop:equ-PADMM-dPPM}, we can directly deduce the convergence of Algorithm \ref{alg:pADMM} with varying relaxation factors $\rho_k \in (0,2)$ for $k\geq 0$ by applying the convergence results of the dPPM, as presented in Theorem \ref{Th: convergence-dPPM}.
	\begin{corollary}
		Suppose that Assumptions \ref{ass: CQ} and \ref{ass: Assump-solvability} hold. If $0<\inf_k \rho_k\leq \sup_k \rho_k<2$, then the sequence $\{w^k\}=\{(y^k,z^k,x^k)\}$ generated by Algorithm \ref{alg:pADMM} converges to the point $w^{*}=(y^*,z^*,x^*)$, where $(y^{*},z^{*})$ is a solution to problem \eqref{primal} and $x^{*}$ is a solution to problem \eqref{dual}.
	\end{corollary}
	\subsection{An accelerated pADMM}
	Based on the equivalence stated in Proposition \ref{prop:equ-PADMM-dPPM}, we can employ the accelerated dPPM introduced in Algorithm \ref{alg:acc-dPPM} to derive an accelerated pADMM, as outlined in Algorithm \ref{alg:acc-pADMM}.
	\begin{algorithm}[ht]
		\caption{An accelerated pADMM for solving two-block convex optimization problem \eqref{primal}}
		\label{alg:acc-pADMM}
		\begin{algorithmic}[1]
			\STATE {Input: Let $\cT_1$ and $\cT_2$ be two self-adjoint positive semidefinite linear operators $\mathbb{Y}$ and $\mathbb{Z}$, respectively. Choose an initial point $w^{0}=(y^{0}, z^{0}, x^{0})\in \operatorname{dom} (f_1) \times \operatorname{dom} (f_2)\times \mathbb{X}$. Let $\hat{w}^0:=w^{0}$. Set parameters $\sigma > 0$, $\alpha\geq 2$ and $\rho\in(0,2]$. For $k=0,1, \ldots,$ perform the following steps in each iteration.}
			\STATE {Step 1. $\bz^{k}=\underset{z \in \mathbb{Z}}{\arg \min }\left\{L_\sigma\left(y^k, z ; x^k\right)+\frac{1}{2}\|z-z^{k}\|_{\mathcal{T}_2}^2\right\}$.}
			\STATE{Step 2. $\bx^{k}={x}^k+\sigma (B_{1}{y}^{k}+B_{2}\bz^{k}-c) $.}
			\STATE {Step 3. $\by^{k}=\underset{y \in \mathbb{Y}}{\arg \min }\left\{L_\sigma\left(y, \bz^{k} ; \bx^{k}\right)+\frac{1}{2}\|y-y^{k}\|_{\mathcal{T}_1}^2\right \}$.
			}
			\STATE {Step 4. $\hw^{k+1}= (1-\rho){w}^{k}+ \rho \bw^{k}$.}
			\STATE {Step 5. $w^{k+1}=w^k +\frac{\alpha}{2(k+\alpha)} (\hat{w}^{k+1}-w^k)+\frac{k}{k+\alpha}\left(\hat{w}^{k+1}-\hat{w}^{k}\right)$.}
		\end{algorithmic}
	\end{algorithm}
	\begin{remark}
		{When $\cT_i=0$ for $i=1,2$ in Algorithm \ref{alg:acc-pADMM}, we can obtain an accelerated ADMM. As will be shown in Theorem \ref{Th:complexity-acc-pADMM}, this accelerated ADMM exhibits a nonergodic convergence rate of $O(1/k)$ for both the KKT residual and the primal objective function value gap, which is similar to the findings in \cite{zhang2022efficient}.  Furthermore, by setting $\cT_i=\sigma(\lambda_{\max}(B_i^*B_i)\cI-B_i^*B_i)$ for $i=1,2$ in Algorithm \ref{alg:acc-pADMM}, we can obtain an accelerated LADMM. Compared to the algorithm in \cite{li2019accelerated}, the $\cT_i$ for $i=1,2$ in Algorithm \ref{alg:acc-pADMM} will not tend to infinity as $k$ increases, which implies that this accelerated LADMM has a larger primal step length. Additionally, both $\cT_1$ and $\cT_2$ in Algorithm \ref{alg:acc-pADMM} can be positive semidefinite under Assumption \ref{ass: Assump-solvability}, which is a significant difference compared to work of \cite{sabach2022faster}. Finally, the accelerated pADMM introduced in \cite{yang2023accelerated}, where both $\mathcal{T}_1$ and $\mathcal{T}_2$ are positive definite, is a special case of Algorithm \ref{alg:acc-pADMM} with $\alpha = 2$.}
	\end{remark}
	
	According to the convergence result of the accelerated dPPM in Theorem \ref{Th:convergence-Halpern-dPPM}, we can obtain the global convergence of Algorithm \ref{alg:acc-pADMM} in the following corollary.
	\begin{corollary}\label{coro:convergence-acc-pADMM}
		Suppose that Assumptions \ref{ass: CQ} and \ref{ass: Assump-solvability} hold. The sequence $\{\bw^k\}=\{(\by^k,\bz^k,\bx^k)\}$ generated by Algorithm \ref{alg:acc-pADMM} {converges to} the point $w^{*}=(y^*,z^*,x^*)$, where $(y^{*},z^{*})$ is a solution to problem \eqref{primal} and $x^{*}$ is a solution to problem \eqref{dual}.
	\end{corollary}
	\begin{proof}
		It follows from Theorem \ref{Th:convergence-Halpern-dPPM} and Proposition \ref{prop:equ-PADMM-dPPM}.
	\end{proof}
	To analyze the convergence rate of Algorithm \ref{alg:acc-pADMM}, we begin by considering the residual mapping associated with the KKT system \eqref{eq:KKT}, as introduced in \cite{han2018linear}:
	\begin{equation}\label{def:KKT_residual}
		\mathcal{R}(w): = \left(
		\begin{array}{c}
			y - {\rm Prox}_{f_1}(y - B_1^*x)\\
			z - {\rm Prox}_{f_2}(z - B_2^*x)\\
			c - B_1 y - B_2 z
		\end{array}
		\right), \quad \forall w=(y,z,x) \in \mathbb{W}.
	\end{equation}
	It is clear that $w^{*}=(y^*, z^*, x^*)$ satisfies the KKT system \eqref{eq:KKT} if and only if $\mathcal{R}(w^{*}) = 0$. Let $\{(\by^k,\bz^k)\}$ be the sequence generated by Algorithm \ref{alg:acc-pADMM}. To estimate the primal objective function value gap, we define
	$$
	h(\by^{k}, \bz^{k}):=f_1(\by^{k})+f_2(\bz^{k})-f_1(y^*)-f_2(z^*), \quad \forall k \geq 0,
	$$
	where $\left(y^*, z^*\right)$ is the limit point of the sequence $\{(\by^k, \bz^k)\}$. The next lemma provides the lower and upper bounds for the primal objective function value gap.
	\begin{lemma}
		Suppose that Assumptions \ref{ass: CQ} and \ref{ass: Assump-solvability} hold. Let $\{(\by^{k},\bz^{k},\bx^{k})\}$ be the sequence generated by Algorithm \ref{alg:acc-pADMM}. Then for all $k \geq 0$, we have the following bounds:
		\begin{equation}\label{obj-upperbound}
			\begin{array}{ll}
				h(\by^k,\bz^k) & \leq  \left\langle\sigma B_1( y^*-\by^k)-\bx^k,  (B_1 \by^k+B_2\bz^k-c )\right \rangle\\
				& \quad + \left\langle y^*-\by^k,  \cT_1(\by^k-y^k)\right\rangle +\left\langle z^*-\bz^k , \cT_2(\bz^k-z^k)\right\rangle
			\end{array}
		\end{equation}
		and
		\begin{equation}\label{obj-lowerbound}
			\begin{array}{l}
				h\left(\by^{k}, \bz^{k}\right) \geq\left\langle B_1 \by^{k}+B_2 \bz^{k}-c,-x^*\right\rangle,
			\end{array}
		\end{equation}
		where $\left(y^*, z^*,x^*\right)$ is the limit point of the sequence $\{(\by^k, \bz^k,\bx^k)\}$.
	\end{lemma}
	
	\begin{proof}
		For notational convenience, for any $x \in \mathbb{X}$, define
		$$
		d_{f_1}(x):=f_1^*\left(-B_1^* x\right)+\langle x, c\rangle, \quad  \quad d_{f_2}(x):=f_2^*\left(-B_2^* x\right)
		$$
		and
		$$
		\tx^k:=\bx^k+\sigma (B_1 \by^k+B_2\bz^k-c), \quad  \forall k \geq 0.
		$$
		On the one hand, from the optimality conditions of the subproblems of Algorithm \ref{alg:acc-pADMM}, we have for any $k\geq 0$,
		$$
		-B_2^* \bx^{k}-\cT_2(\bz^k-z^k) \in \partial f_2(\bz^k), \quad-B_1^* \tx^k- \cT_1(\by^k-y^k)\in \partial f_1(\by^k).
		$$
		It follows from \cite[Theorem 23.5]{rockafellar1970convex} that for any $k \geq 0$,
		$$
		\left\{
		\begin{array}{l}
			f_1^*\left(-B_1^* \tx^k- \cT_1(\by^k-y^k)\right)=\left\langle \by^k, -B_1^* \tx^k- \cT_1(\by^k-y^k)\right\rangle-f_1\left(\by^k\right),\\
			f_2^*\left(-B_2^* \bx^{k}-\cT_2(\bz^k-z^k)\right)=\left\langle \bz^k, -B_2^* \bx^{k}-\cT_2(\bz^k-z^k)\right\rangle-f_2\left(\bz^{k}\right) .
		\end{array}\right.
		$$
		Summing them up, we can obtain that for any $k\geq 0$,
		\begin{equation}\label{obj-upperbound-1}
			\begin{array}{ll}
				&f_1^*\left(-B_1^* \tx^k- \cT_1(\by^k-y^k)\right)+f_2^*\left(-B_2^* \bx^{k}-\cT_2(\bz^k-z^k)\right)+\langle \tx^k, c\rangle\\
				=&-f_1\left(\by^k\right)-f_2\left(\bz^{k}\right)-   \left\langle   (B_1 \by^k+B_2\bz^k-c), \bx^k+\sigma (B_1\by^k-c)\right\rangle \\
				&-\left\langle \by^k, \cT_1(\by^k-y^k)\right\rangle-\left\langle \bz^k, \cT_2(\bz^k-z^k)\right\rangle.
			\end{array}
		\end{equation}
		On the other hand, from the KKT conditions in \eqref{eq:KKT} and \cite[Theorem 23.5]{rockafellar1970convex}, we have
		$$
		\begin{array}{l}
			y^*\in \partial f_{1}^{*}(-B_1^{*}x^*), \quad z^*\in \partial f_{2}^{*}(-B_2^{*}x^*).\\
		\end{array}
		$$
		Thus, it follows from the convexity of $f_1^*$ and $f_2^*$ that for any $k\geq 0$,
		$$
		\left\{
		\begin{array}{ll}
			&f_1^*\left(-B_1^* \tx^k- \cT_1(\by^k-y^k)\right)+\langle \tx^k, c\rangle-d_{f_1}(x^*) \\
			&\quad\geq \left\langle y^*,  -B_1^* \tx^k- \cT_1(\by^k-y^k)+B_1^*x^*\right\rangle+\langle c,\tx^k-x^* \rangle,  \\
			&f_2^*\left(-B_2^* \bx^{k}-\cT_2(\bz^k-z^k)\right)-d_{f_2}(x^*)\geq \left\langle z^* , -B_2^* \bx^{k}-\cT_2(\bz^k-z^k)+B_2^{*}x^{*}\right\rangle.\\
		\end{array}\right.
		$$
		Summing them up and by noting $B_1y^*+B_2{z^*}=c$, we have for any $k\geq0$,
			\begin{equation}\label{obj-upperbound-2}
				\begin{array}{l}
					f_1^*\left(-B_1^* \tx^k- \cT_1(\by^k-y^k)\right)+f_2^*\left(-B_2^* \bx^{k}-\cT_2(\bz^k-z^k)\right)+\langle \tx^k, c\rangle \\
					\quad \geq d_{f_1}(x^*)+d_{f_2}(x^*) -\left\langle B_1 y^* -c, \sigma (B_1 \by^k+B_2\bz^k-c )\right \rangle  \\
					\qquad -\left\langle y^*,  \cT_1(\by^k-y^k)\right\rangle -\left\langle z^* , \cT_2(\bz^k-z^k)\right\rangle.
				\end{array}
			\end{equation}
			Combing \eqref{obj-upperbound-1} and \eqref{obj-upperbound-2}, we can deduce that for any $k\geq 0$,
			\begin{equation*}
				\begin{array}{ll}				&f_1\left(\by^k\right)+f_2\left(\bz^k\right)+d_{f_1}(x^*)+d_{f_2}(x^*)\\ &\quad \leq       \left\langle\sigma B_1( y^*-\by^k)-\bx^k,  (B_1 \by^k+B_2\bz^k-c )\right \rangle \\
					&\qquad   + \left\langle y^*-\by^k,  \cT_1(\by^k-y^k)\right\rangle +\left\langle z^*-\bz^k , \cT_2(\bz^k-z^k)\right\rangle.
				\end{array}
			\end{equation*}
			Since  $d_{f_1}(x^*)+d_{f_2}(x^*)=-f_1(y^*)-f_2(z^*)$ from \cite[Theorem 28.4]{rockafellar1970convex}, we can derive for any $k\geq0$,
			\begin{equation*}
				\begin{array}{ll}
					h(\by^k,\bz^k) &\leq  \left\langle\sigma B_1( y^*-\by^k)-\bx^k,  (B_1 \by^k+B_2\bz^k-c )\right \rangle\\
					&\quad+ \left\langle y^*-\by^k,  \cT_1(\by^k-y^k)\right\rangle +\left\langle z^*-\bz^k , \cT_2(\bz^k-z^k)\right\rangle.
				\end{array}
			\end{equation*}
			Finally, according to the KKT conditions in \eqref{eq:KKT}, we have for any $k\geq 0$,
			$$
			f_1(\by^k)-f_1(y^*) \geq\langle-B_1^* x^*, \by^{k}-y^*\rangle, \quad f_2(\bz^{k})-f_2\left(z^*\right) \geq\langle {-{B_2^*}} x^*, \bz^{k}-z^*\rangle .
			$$
			Thus, the inequality \eqref{obj-lowerbound} holds since $B_1 y^*+B_2 z^*=c$. This completes the proof of the lemma.
		\end{proof}
		Now, we are ready to present the complexity result for Algorithm \ref{alg:acc-pADMM}.
		\begin{theorem}\label{Th:complexity-acc-pADMM}
			Suppose that Assumptions \ref{ass: CQ} and \ref{ass: Assump-solvability} hold. Let $\{(\by^{k},\bz^{k},\bx^{k})\}$ be the sequence generated by Algorithm \ref{alg:acc-pADMM}, and let $w^*=(y^*,z^*,x^*)$ be the limit point of the sequence $\{(\by^{k},\bz^{k},\bx^{k})\}$ and $R_0=\|w^{0}-w^{*}\|_{\cM}$.
			\begin{enumerate}
				\item[(a)] If $\alpha =2$, then for all $k \geq 0$, we have the following bounds:
				\begin{equation}
					\label{eq: complexity-bound-KKT=2}
					\|\mathcal{R}(\bar w^k)\| \leq \left( \frac{\sigma \|B_1^{*}\|+1}{\sqrt{\sigma}}+\|\sqrt{\cT_2}\|+\|\sqrt{\cT_1}\| \right) \frac{2R_0}{\rho(k+1)}
				\end{equation}
				and
				\begin{equation}
					\begin{array}{ll}
						\left(\frac{-1}{\sqrt{\sigma}}\|x^*\|\right) \frac{2 R_0}{\rho(k+1)}\leq h(\by^k,\bz^k)  \leq  \left(3R_0  + \frac{1}{\sqrt{\sigma}}\|x^*\|\right) \
						\frac{2R_0}{\rho(k+1)}.
					\end{array}
				\end{equation}
				\item[(b)] If $\alpha >2$, then we have the following bounds:
				\begin{equation}
					\begin{array}{ll}\label{eq: complexity-bound-KKT>2}
						\|\mathcal{R}(\bar w^k)\| =\left( \frac{\sigma \|B_1^{*}\|+1}{\sqrt{\sigma}}+\|\sqrt{\cT_2}\|+\|\sqrt{\cT_1}\| \right) o\left(\frac{1}{k+1}\right)\text { as } k \rightarrow+\infty \text {}
					\end{array}
				\end{equation}
				and
				\begin{equation}
					\begin{aligned}
						|h(\by^{k}, \bz^{k})|= o\left(\frac{1}{k+1}\right)\text { as } k \rightarrow+\infty \text {. }
					\end{aligned}
				\end{equation}
			\end{enumerate}
		\end{theorem}
		\begin{proof}
			Consider the case where $\alpha=2$. We first estimate the convergence rate of $\mathcal{R}(\bw^k)$ for any $k\geq 0$. According to Proposition \ref{prop:complexity-acc-dPPM}, we have
			$$
			\|\hat{w}^{k+1}-w^{k}\|_{\cM}^{2}\leq \frac{4R_0^2}{(k+1)^2}, \quad \forall k\geq 0.
			$$
			By the definition of $\cM$ in \eqref{def:M}, this can be rewritten as
			\begin{equation}\label{Th:complexity-acc-pADMM-1}
				\|\hy^{k+1}-y^{k}\|^{2}_{\cT_1}+\frac{1}{\sigma}\|\sigma B_{1}(\hat{y}^{k+1}-y^{k})+(\hx^{k+1}-x^{k})\|^{2}+\|\hz^{k+1}-z^{k}\|_{\cT_2}^2 \leq \frac{4R_0^2}{(k+1)^2}, \forall k\geq 0.
			\end{equation}
			Also, from Step 4 in Algorithm \ref{alg:acc-pADMM}, we can obtain that for any $k\geq0$,
			\begin{equation*}
				\left\{
				\begin{array}{c}
					\hy^{k+1}-y^{k}=\rho(\by^{k}-y^{k}), \\
					\hz^{k+1}-z^{k}=\rho(\bz^{k}-z^{k}), \\
					\hx^{k+1}-x^{k}=\rho(\bx^{k}-x^{k}).
				\end{array}
				\right.
			\end{equation*}
			Thus, we can rewrite \eqref{Th:complexity-acc-pADMM-1} as
			\begin{equation}\label{Th:complexity-acc-pADMM-2}
				\|\by^k-y^{k}\|^{2}_{\cT_1}+\frac{1}{\sigma}\|\sigma B_{1}(\by^{k}-y^{k})+(\bx^{k}-x^{k})\|^{2}+\|\bz^{k}-z^{k}\|_{\cT_2}^2 \leq \frac{4R_0^2}{\rho^2(k+1)^2}, \forall k\geq 0.
			\end{equation}
			Due to Step 2 in Algorithm \ref{alg:acc-pADMM}, we can deduce that for any $k\geq 0$,
			\begin{equation*}
				\begin{array}{ll}
					\|\sigma  B_{1}(\by^{k}-y^{k})+(\bx^{k}-x^{k})\|&=\|\sigma   B_{1}(\by^{k}-y^{k})+\sigma (B_{1}{y}^{k}+B_{2}\bz^{k}-c)\|\\
					&=\sigma\| B_{1}{\by}^{k}+B_{2}\bz^{k}-c\|,
				\end{array}
			\end{equation*}
			which together with \eqref{Th:complexity-acc-pADMM-2} yields that
			\begin{equation}\label{Th:complexity-acc-pADMM-3}
				\| B_{1}{\by}^{k}+B_{2}\bz^{k}-c\| \leq  \frac{2 R_0}{\sqrt{\sigma}\rho(k+1)}, \quad \forall k\geq 0.
			\end{equation}
			Moreover, from the optimality conditions of the subproblems in Algorithm \ref{alg:acc-pADMM}, we have for any $k\geq 0$,
			\begin{equation}\label{Th:complexity-acc-pADMM-4}
				\left\{
				\begin{array}{l}
					\bz^{k}={\rm Prox}_{f_2}(\bz^{k}-B_{2}^{*}\bx^{k}-\cT_2(\bz^{k}-z^{k})), \\
					\by^{k}={\rm Prox}_{f_1}(\by^{k}-B_{1}^{*}(\bx^{k}+ \sigma (B_{1}{\by}^{k}+B_{2}\bz^{k}-c))           -\cT_1(\by^{k}-y^{k})),
				\end{array}
				\right.
			\end{equation}
			which together with \eqref{Th:complexity-acc-pADMM-2} yields that for any $k\geq 0$,
			\begin{equation}\label{Th:complexity-acc-pADMM-5}
				\begin{array}{ll}
					&\|\bz^{k} - {\rm Prox}_{f_2}(\bz^{k} - B_2^*\bx^{k}) \|\\
					=&\|{\rm Prox}_{f_2}(\bz^{k}-B_{2}^{*}\bx^{k}-\cT_2(\bz^{k}-z^{k}))- {\rm Prox}_{f_2}(\bz^{k} - B_2^*\bx^{k}) \|\\
					\leq& \|\cT_2(\bz^{k}-z^{k})\|\\
					\leq& \|\sqrt{\cT_2}\|\|\bz^{k}-z^{k}\|_{\cT_2}\\
					\leq& \|\sqrt{\cT_2}\|\frac{2R_0}{\rho(k+1)}.
				\end{array}
			\end{equation}
			Similarly,  from \eqref{Th:complexity-acc-pADMM-2}, \eqref{Th:complexity-acc-pADMM-3}, and \eqref{Th:complexity-acc-pADMM-4}, we also have for any $k\geq 0$,
			\begin{equation}\label{Th:complexity-acc-pADMM-6}
				\begin{array}{ll}
					&\|\by^{k} - {\rm Prox}_{f_1}(\by^{k} - B_1^*\bx^{k}) \|\\
					\leq& \|B_1^*\sigma (B_{1}{\by}^{k}+B_{2}\bz^{k}-c) +\cT_1(\by^{k}-y^{k})\|\\
					\leq& \sigma \|B_1^*\| \|B_{1}{\by}^{k}+B_{2}\bz^{k}-c\| +\|\cT_1(\by^{k}-y^{k})\|\\
					\leq&  (\sqrt{\sigma}\|B_1^*\|+ \|\sqrt{\cT_1}\|)\frac{2R_0}{\rho(k+1)}.
				\end{array}
			\end{equation}
			Therefore, by  \eqref{Th:complexity-acc-pADMM-3}, \eqref{Th:complexity-acc-pADMM-5}, and \eqref{Th:complexity-acc-pADMM-6}, we can obtain that for any $k\geq 0$,
			\begin{equation}
				\begin{array}{ll}
					\|\mathcal{R}(\bar w^k)\|&\leq \sqrt{\left( \frac{1}{\sigma}+\|\sqrt{\cT_2}\|^2+ (\sqrt{\sigma} \|B_1^{*}\|+\|\sqrt{\cT_1}\|)^2 \right)} \frac{2R_0}{\rho(k+1)}\\
					&\leq \left( \frac{\sigma \|B_1^{*}\|+1}{\sqrt{\sigma}}+\|\sqrt{\cT_2}\|+\|\sqrt{\cT_1}\| \right) \frac{2R_0}{\rho(k+1)}.
				\end{array}
			\end{equation}
			Now, we estimate the complexity result concerning the primal function value gap. For the lower bound of the primal function value gap, from \eqref{obj-lowerbound} and \eqref{Th:complexity-acc-pADMM-3}, we have for all $k \geq 0$,
			$$
			\begin{aligned}
				h(\by^{k}, \bz^{k}) & \geq\langle B_1 \by^{k}+B_2 \bz^{k}-c,-x^*\rangle \\
				& \geq -\|x^*\|\| B_1\by^{k}+B_2 \bz^{k}-c\|\\
				& \geq -\frac{2 R_0\|x^*\|}{\sqrt{\sigma}\rho(k+1)}.
			\end{aligned}
			$$
			{On the other hand, $\hcF_\rho$ is $\cM$-nonexpansive for $\rho\in(0,2]$ by Proposition \ref{prop:M-nonexpansive}. Similar to \eqref{complexity-HPPPA-5}, we can derive from the induction that}
			\begin{equation*}
				\| w^{k}  -w^{*}\|_{\mathcal{M}}\leq R_0,\quad \forall k \geq 0.
			\end{equation*}
			Hence, from the $\cM$-nonexpansiveness of $\hcT$ by Proposition \ref{prop:M-nonexpansive}, we have
			\begin{equation*}
				\|\bw^{k}  -w^{*}\|_{\mathcal{M}}= \|\hcT w^k  -w^{*}\|_{\mathcal{M}} \leq \| w^k  -w^{*}\|_{\cM} \leq R_0, \quad \forall k \geq 0,
			\end{equation*}
			which implies
			\begin{equation*}\label{Th:complexity-acc-pADMM-8}
				\|\by^{k}-y^{*}\|^{2}_{\cT_1}+\frac{1}{\sigma}\|\sigma B_{1}(\by^{k}-y^{*})+(\bx^{k}-x^{*})\|^{2}+\|\bz^{k}-z^{*}\|_{\cT_2}^2 \leq R_0^2,\quad \forall k \geq 0.
			\end{equation*}
			This inequality together with \eqref{obj-upperbound}, \eqref{Th:complexity-acc-pADMM-2}, and \eqref{Th:complexity-acc-pADMM-3} yields that for all $k\geq0$,
			\begin{equation*}
				\begin{array}{ll}
					h(\by^k,\bz^k) &\leq  (\|\sigma B_1(\by^k-y^*)+(\bx^k-x^*)\| +\|x^*\|) \|B_1 \by^k+B_2\bz^k-c \|\\
					&\quad + \| y^*-\by^k\|_{\cT_1}\|\by^k-y^k\|_{\cT_1} +\| z^*-\bz^k\|_{\cT_2}\|\bz^k-z^k\|_{\cT_2}\\
					&\leq  (\sqrt{\sigma}R_0  +\|x^*\|)\frac{2 R_0}{\sqrt{\sigma}\rho(k+1)} +\frac{4R_0^2}{\rho(k+1)} \\
					&=  \left(3R_0  + \frac{1}{\sqrt{\sigma}}\|x^*\|\right)\frac{2 R_0}{\rho(k+1)}.
				\end{array}
			\end{equation*}
			
			\par Now, we establish the complexity results for the case where $\alpha>2$. According to Proposition \ref{prop:complexity-acc-dPPM}, we have
			\begin{equation*}
				\|w^{k}-\hat{w}^{k+1}\|_{\cM}=o\left(\frac{1}{k+1}\right) \text { as } k \rightarrow+\infty \text {. }
			\end{equation*}
			Similar to the previous case with $\alpha=2$, we can obtain
			\begin{equation*}
				\begin{array}{ll}
					\|\mathcal{R}(\bar w^k)\| =\left( \frac{\sigma \|B_1^{*}\|+1}{\sqrt{\sigma}}+\|\sqrt{\cT_2}\|+\|\sqrt{\cT_1}\| \right) o\left(\frac{1}{k+1}\right)\text { as } k \rightarrow+\infty \text {. }
				\end{array}
			\end{equation*}
			In addition, since $\|\bw^k-w^*\|$ is bounded due to Corollary \ref{coro:convergence-acc-pADMM}, from \eqref{obj-upperbound}, there exists a constant $C_1$ such that
			\begin{equation*}
				\begin{array}{ll}
					h(\by^k,\bz^k) &\leq  C_1 \left(\|B_1 \by^k+B_2\bz^k-c \|+\|\by^k-y^k\|_{\cT_1}+\|\bz^k-z^k\|_{\cT_2}\right).
				\end{array}
			\end{equation*}
			Hence, following the proof of the previous case with $\alpha=2$, we can similarly demonstrate that
			\begin{equation*}
				\begin{aligned}
					|h(\by^{k}, \bz^{k})|= o\left(\frac{1}{k+1}\right)\text { as } k \rightarrow+\infty \text {,}
				\end{aligned}
			\end{equation*}
			which completes the proof.
		\end{proof}
		
		{
			\begin{remark}
				For convex optimization problem \eqref{primal}, $f_1(\cdot)$ and $f_2(\cdot)$ are often of a composite structure, i.e. $f_1=g_1+p_1$, and $f_2=g_2+p_2$, where $g_1:\mathbb{Y}\to (-\infty, +\infty)$ and $g_2:\mathbb{Z}\to (-\infty, +\infty)$ are two continuously differentiable convex functions (e.g., convex quadratic functions), $p_{1}: \mathbb{Y} \to (-\infty, +\infty]$ and $p_{2}: \mathbb{Z} \to (-\infty, +\infty]$ are two proper closed convex functions. In this case, it is more convenient for computational purposes to define the following KKT residual mapping:
				\begin{equation}\label{def:KKT_residual-2}
					\mathcal{\widetilde{R}}(w) := \left(
					\begin{array}{c}
						y - {\rm Prox}_{p_1}(y-\nabla g_1(y) - B_1^*x)\\
						z - {\rm Prox}_{p_2}(z-\nabla g_2(z) - B_2^*x)\\
						c - B_1 y - B_2 z
					\end{array}
					\right), \quad \forall w=(y,z,x) \in \mathbb{W}.
				\end{equation}
				By substituting $\mathcal{\widetilde{R}}(\cdot)$ into \eqref{Th:complexity-acc-pADMM-5} and \eqref{Th:complexity-acc-pADMM-6}, and replacing \eqref{Th:complexity-acc-pADMM-4} with the following relations:
				\begin{equation*}
					\left\{
					\begin{array}{l}
						\bz^{k} = {\rm Prox}_{p_2}(\bz^{k} - \nabla g_2(\bz^k) - B_2^{*}\bx^{k} - \cT_2(\bz^{k} - z^{k})), \\
						\by^{k} = {\rm Prox}_{p_1}(\by^{k} - \nabla g_1(\by^k) - B_1^{*}(\bx^{k} + \sigma(B_1\by^{k} + B_2\bz^{k} - c)) - \cT_1(\by^{k} - y^{k})),
					\end{array}
					\right.
				\end{equation*}
				one can verify in a straightforward way that the iteration complexity results in Theorem \ref{Th:complexity-acc-pADMM} remain valid for $\mathcal{\widetilde{R}}(\cdot)$. In the numerical experiments for solving convex QP problems, we use $\mathcal{\widetilde{R}}(\cdot)$ in \eqref{def:KKT_residual-2} instead of $\mathcal{R}(\cdot)$ in \eqref{def:KKT_residual} to construct the stopping criterion.

			\end{remark}
		}
		
		\section{Numerical experiments}\label{Sec: Numerical}
		In this section, we will utilize the convex {QP} problem as an illustrative case to compare the performance of the pADMM in Algorithm \ref{alg:pADMM} and the accelerated pADMM (acc-pADMM) in Algorithm \ref{alg:acc-pADMM}. Specifically, the high-dimensional convex QP problems can be formulated in the following standard form:
		\begin{equation}\label{QP-primal}
			\min _{x \in \mathbb{R}^n}\left\{\frac{1}{2}\langle x, Q x\rangle+\langle c, x\rangle \mid A x=b, x \in C\right\},
		\end{equation}
		where $c \in \mathbb{R}^n$, $b \in \mathbb{R}^m$, and $C=\left\{x \in \mathbb{R}^n: l \leq x \leq u\right\}$, with $\ell, u \in \mathbb{R}^n$ being given vectors satisfying $-\infty \leq l \leq u \leq +\infty$. Furthermore,  $A \in \mathbb{R}^{m \times n}$ has full row rank, and $Q\in  \mathbb{R}^{n \times n}$ is a symmetric positive semidefinite matrix. The corresponding restricted-Wolfe dual \cite{li2018qsdpnal} of problem \eqref{QP-primal} can be expressed as follows:
		\begin{equation}\label{QP-dual}
			\min _{(z_1,z_2,y ) \in   \mathbb{R}^n \times \mathbb{R}^{m} \times \mathbb{R}^{n}}\left\{\frac{1}{2}\langle y, Q y\rangle +\delta_C^*(-z_1)  -\langle b, z_2\rangle \mid -Q y+z_1+A^*z_2=c, \ y \in \mathcal{Y}\right\},
		\end{equation}
		where $\mathcal{Y}$ represents any subspace of $\mathbb{R}^n$ containing {the range space of $Q$ denoted by $\operatorname{Range}(Q)$}, and $\delta_C^*(\cdot)$ represents the convex conjugate of the indicator function $\delta_C(\cdot)$ of the set $C$. In this article, we assume that $\mathcal{Y}=\operatorname{Range}(Q)$. Let $y\in \mathbb{R}^{n}, z=(z_1,z_2)\in \mathbb{R}^{n} \times \mathbb{R}^{m}$, $B_1=-Q$, $B_2=[I,A^*]$, $f_1(y)=\frac{1}{2}\langle y, Q y\rangle$, and $f_2(z)=\delta_C^*(-z_1)-\langle b, z_2\rangle$. Consequently, problem \eqref{QP-dual} can be reformulated as a two-block convex optimization problem in \eqref{primal}.

		Given $\sigma>0$, we define the augmented Lagrangian function $L_\sigma(y, z_1, z_2 ; x)$ associated with problem \eqref{QP-dual} for any $(y, z_1, z_2, x) \in \mathcal{Y} \times \mathbb{R}^{n} \times \mathbb{R}^{m} \times \mathbb{R}^n$ as follows:
		$$\begin{array}{ll}
			L_\sigma(y, z_1, z_2 ; x)&=\frac{1}{2}\langle y, Q y\rangle+\delta_C^*(-z_1)-\langle b, z_2\rangle\\
			&\qquad +\frac{\sigma}{2}\left\|-Q y+z_1+A^*z_2-c+\sigma^{-1} x\right\|^2-\frac{1}{2 \sigma}\|x\|^2.
		\end{array}
		$$
		To simplify the solution of subproblems involving the variable $z = (z_1, z_2)$, we introduce a specific operator $\mathcal{T}_2$ called the {sGS} operator \cite{li2016schur,li2019block}. Specifically, given $\sigma > 0$, we define a self-adjoint linear operator $\widetilde{B}:\mathbb{R}^{n}\times \mathbb{R}^{m}\rightarrow\mathbb{R}^{n}\times \mathbb{R}^{m}$ as follows:
		$$
		\widetilde{B}:=\sigma B_2^{*}B_2=\sigma\left[ \begin{array}{cc}
			I &  A^{*}  \\
			A&   AA^{*}
		\end{array}   \right].
		$$
		Then we decompose $\widetilde{B}$ as:
		$$
		\widetilde{B}=U_{\widetilde{B}}+D_{\widetilde{B}}+U_{\widetilde{B}}^{*},
		$$
		where
		$$
		U_{\widetilde{B}}=\sigma\left[ \begin{array}{cc}
			0 & A^{*}   \\
			0& 0
		\end{array}   \right], \quad D_{\widetilde{B}}=\sigma\left[ \begin{array}{cc}
			I & 0  \\
			0& AA^{*}
		\end{array}   \right].
		$$
		The sGS operator associated with $\widetilde{B}$ denoted by $\mathrm{sGS}(\widetilde{B}):\mathbb{R}^{n}\times \mathbb{R}^{m}\rightarrow\mathbb{R}^{n}\times \mathbb{R}^{m}$ can be defined as follows:
		$$
		\mathrm{sGS}(\widetilde{B})=U_{\widetilde{B}} D_{\widetilde{B}}^{-1} U_{\widetilde{B}}^{*}.
		$$
		Let $\mathcal{T}_1=0$ and $\mathcal{T}_2=\mathrm{sGS}(\widetilde{B})$. We can obtain
		$$
		\sigma B_1^*B_1 + \mathcal{T}_1 + \Sigma_{f_1} = \sigma Q^2 + Q,
		$$
		which is positive definite in $\mathcal{Y}$. Furthermore, according to \cite[Theorem 1]{li2019block}, we also know that
		$$
		\sigma B_2^*B_2 + \mathcal{T}_2 + \Sigma_{f_2} = \widetilde{B} + \mathrm{sGS}(\widetilde{B})
		$$
		is positive definite. {As a result, we can {apply the} acc-pADMM, presented in Algorithm \ref{alg:acc-pADMM-QP}, to {solve} problem \eqref{QP-dual}, thereby {also} effectively solving problem \eqref{QP-primal}.} Note that according to \cite[Thoerem 1]{li2019block}, Step 1 (1-2-3) in Algorithm \ref{alg:acc-pADMM-QP} is equivalent to
		$$
		(\bz_1^k,\bz_2^k)=\underset{ (z_1,z_2) \in \mathbb{R}^{n}\times \mathbb{R}^{m}}{\arg \min }\left\{L_\sigma(y^k, z_1, z_2 ; x^k)+\frac{1}{2}\|z-z^k\|_{\mathrm{sGS}(\widetilde{B})}^2\right\}.
		$$
		In addition, by omitting Step 5 in Algorithm \ref{alg:acc-pADMM-QP} and choosing $\rho\in (0,2)$, we can derive a pADMM for solving problem \eqref{QP-dual}.
		\begin{algorithm}[ht]
			\caption{An accelerated pADMM for solving problem \eqref{QP-dual}}
			\label{alg:acc-pADMM-QP}
			\begin{algorithmic}[1]
				\STATE {Input: Choose $w^{0}=(y^0,z_1^0,z_2^0,x^0)\in   \mathcal{Y} \times \mathbb{R}^{n} \times \mathbb{R}^{m} \times   \mathbb{R}^{n}$. Let $\hat{w}^0:=w^{0}$. Set parameters $\sigma > 0$, $\alpha\geq 2$ and $\rho\in(0,2]$. For $k=0,1, \ldots,$ perform the following steps in each iteration.}
				\STATE {Step 1-1. $\bz_2^{k\prime}=\underset{z_2 \in \mathbb{R}^{m}}{\arg \min }\left\{L_\sigma(y^k, z_1^k, z_2 ; x^k)\right\}$.}
				\STATE {Step 1-2. $\bz_1^{k}=\underset{z_1 \in \mathbb{R}^{n}}{\arg \min }\left\{L_\sigma(y^k, z_1, \bz_2^{k\prime}; x^k)\right\}$.}
				\STATE {Step 1-3. $\bz_2^{k}=\underset{z_2 \in \mathbb{R}^{m}}{\arg \min }\left\{L_\sigma(y^k, \bz_1^k, z_2 ; x^k)\right\}$.}
				\STATE {Step 2. $\bx^{k}={x}^k+\sigma (-Qy^{k}+ \bz_1^k +A^*{\bz_2}^{k}-c) $.}
				\STATE {Step 3. $\by^{k}=\underset{y \in \mathcal{Y}}{\arg \min }\left\{L_\sigma(y, \bz_1^k, \bz_2^k ; \bx^k)\right\}$.}
				\STATE {Step 4. $\hw^{k+1}= (1-\rho){w}^{k}+ \rho \bw^{k}$.}
				\STATE {Step 5. $w^{k+1}=w^k +\frac{\alpha}{2(k+\alpha)} (\hat{w}^{k+1}-w^k)+\frac{k}{k+\alpha}\left(\hat{w}^{k+1}-\hat{w}^{k}\right)$.}
			\end{algorithmic}
		\end{algorithm}	
		
		{In our experiments, we analyze the performance of all {tested} algorithms using a collection of {25} QP problems selected from the Maros-Mészáros repository \cite{maros1999repository}. Since the data in the Maros-Mészáros repository is relatively sparse, we use the sparse Cholesky decomposition to obtain solutions for the linear systems encountered in the subproblems.} In our implementation, we adopt the following stopping criterion based on the relative KKT residual for the acc-pADMM and the pADMM:
		$$
		\mathrm{KKT}_{\mathrm{res}}=\max \left\{\begin{array}{l}\frac{\left\|A x-b\right\|}{1+\|b\|}, \frac{\left\|-Q y+z_1+A^*z_2-c\right\|}{1+\|c\|}, \frac{\left\|Q x-Q y\right\|}{1+\left\|Q x\right\|+\left\|Q y\right\|}, \frac{\left\|x-\Pi_C\left(x-z_1\right)\right\|}{1+\left\|x\right\|+\left\|z_1\right\|} \end{array} \right\} \leq 10^{-5}.
		$$
		We check the $\mathrm{KKT}_{\mathrm{res}}$ every 50 steps and choose $\rho=1.9$ for the pADMM and $\rho=2$ for the acc-pADMM. Recognizing the sensitivity of algorithm performance to parameter $\sigma$, we employ an adaptive adjustment strategy for $\sigma$ similar to \cite{liang2022qppal} to mitigate its impact. We expect that each algorithm can adaptively select an appropriate $\sigma$. Furthermore, we have observed that restarting significantly enhances the performance of acc-pADMM. Therefore, in this paper, we will restart the acc-pADMM every 200 steps or whenever $\sigma$ varies. Finally, we set the maximum number of iterations as 10000 for all {tested} algorithms. All the numerical experiments in this paper are obtained by running Matlab R2022b on a desktop with Intel(R) Core i9-12900HQ CPU @2.40GHz.
		
		We summarize the computational results for {the tested} 25 problems in Table \ref{tab:my-table-1}. For detailed numerical results of each problem, one may refer to Table \ref{tab:my-table}. {Here, we highlight some key observations}: (1) Firstly, the acc-pADMM demonstrates superior performance compared to the pADMM in the majority of examples. Specifically, the acc-pADMM reduces the number of iterations by an average of 60\% compared to the pADMM. {Moreover}, while the acc-pADMM takes slightly more time per iteration than the pADMM, the acc-pADMM with $\alpha=15$ can still reduce the average time by 47\% compared to the pADMM. (2) Secondly, {the acceleration effect of the acc-pADMM becomes increasingly significant as the iteration number of the pADMM increases. Notably, the acc-pADMM demonstrates its ability to successfully solve challenging problems that pose difficulties for the pADMM, as evidenced by the results of solving problems 8, 11, and 16 in Table \ref{tab:my-table}. In contrast, when the iteration number of the pADMM is very small, the acceleration effect is generally less noticeable, and there are instances where even the acc-pADMM performs worse than the pADMM, as seen in the results of solving problems 1, 2, and 7 in Table \ref{tab:my-table}.} (3) Thirdly, the parameter $\alpha$ also plays a crucial role in the performance of the acc-pADMM. As shown in Table \ref{tab:my-table-1}, we observe that the iteration number of the acc-pADMM decreases as $\alpha$ increases within an appropriate interval. However, when $\alpha$ exceeds a certain threshold, the iteration number of the acc-pADMM starts to increase. To visualize this observation, we present the variations of $\mathrm{KKT}_{\mathrm{res}}$ as the acc-pADMM solves problem 4 with different choices of $\alpha$ in Figure \ref{Fig:KKT_alpha}. Specifically, it is evident from Figure \ref{Fig:KKT_alpha} that the acc-pADMM with $\alpha=30$ achieves  {a satisfactory solution} in the fewest iterations. Conversely, when $\alpha>30$, the acc-pADMM requires more iterations to return a  {satisfactory solution}.		
		
		\begin{table}[H]
			\caption{A comparison between the pADMM and the acc-pADMM for solving 25 QP problems selected from the Maros-Mészáros collection. In the table, ``a", ``$b_{2}$", ``$b_{15}$", ``$b_{30}$", and ``$b_{45}$" denote the pADMM, and the acc-pADMM with $\alpha=2,15,30,45$, respectively.}
			\label{tab:my-table-1}
			\resizebox{1.0\textwidth}{!}{
				\renewcommand\arraystretch{1.5}
				\begin{tabular}{|c|ccccc|ccccc|ccccc|}
					\hline
					& \multicolumn{5}{c|}{Iter}                                                                                                             & \multicolumn{5}{c|}{Time (s)}                                                                                                         & \multicolumn{5}{c|}{$\mathrm{KKT}_{\mathrm{res}}$}                                                                                                            \\ \hline
					25 problems    & \multicolumn{1}{c|}{a}      & \multicolumn{1}{c|}{$b_{2}$} & \multicolumn{1}{c|}{$b_{15}$} & \multicolumn{1}{c|}{$b_{30}$} & $b_{45}$ & \multicolumn{1}{c|}{a}      & \multicolumn{1}{c|}{$b_{2}$} & \multicolumn{1}{c|}{$b_{15}$} & \multicolumn{1}{c|}{$b_{30}$} & $b_{45}$ & \multicolumn{1}{c|}{a}        & \multicolumn{1}{c|}{$b_{2}$}  & \multicolumn{1}{c|}{$b_{15}$} & \multicolumn{1}{c|}{$b_{30}$} & $b_{45}$ \\ \hline
					Mean           & \multicolumn{1}{c|}{3086}   & \multicolumn{1}{c|}{934}     & \multicolumn{1}{c|}{912}      & \multicolumn{1}{c|}{1128}     & 1170     & \multicolumn{1}{c|}{6.44}   & \multicolumn{1}{c|}{3.46}    & \multicolumn{1}{c|}{3.35}     & \multicolumn{1}{c|}{4.33}     & 5.16     & \multicolumn{1}{c|}{3.27E-05} & \multicolumn{1}{c|}{5.06E-06} & \multicolumn{1}{c|}{5.30E-06} & \multicolumn{1}{c|}{5.11E-06} & 1.87E-05 \\ \hline
					Standard Error & \multicolumn{1}{c|}{689.17} & \multicolumn{1}{c|}{254.29}  & \multicolumn{1}{c|}{314.20}   & \multicolumn{1}{c|}{403.24}   & 411.83   & \multicolumn{1}{c|}{5.01}   & \multicolumn{1}{c|}{2.39}    & \multicolumn{1}{c|}{2.34}     & \multicolumn{1}{c|}{3.26}     & 3.94     & \multicolumn{1}{c|}{1.49E-05} & \multicolumn{1}{c|}{5.53E-07} & \multicolumn{1}{c|}{7.02E-07} & \multicolumn{1}{c|}{6.32E-07} & 1.33E-05 \\ \hline
					Median         & \multicolumn{1}{c|}{1750}   & \multicolumn{1}{c|}{450}     & \multicolumn{1}{c|}{350}      & \multicolumn{1}{c|}{350}      & 350      & \multicolumn{1}{c|}{0.18}   & \multicolumn{1}{c|}{0.06}    & \multicolumn{1}{c|}{0.06}     & \multicolumn{1}{c|}{0.06}     & 0.05     & \multicolumn{1}{c|}{9.02E-06} & \multicolumn{1}{c|}{5.14E-06} & \multicolumn{1}{c|}{6.22E-06} & \multicolumn{1}{c|}{4.64E-06} & 5.89E-06 \\ \hline
					Minimum        & \multicolumn{1}{c|}{50}     & \multicolumn{1}{c|}{150}     & \multicolumn{1}{c|}{100}      & \multicolumn{1}{c|}{50}       & 50       & \multicolumn{1}{c|}{0.01}   & \multicolumn{1}{c|}{0.00}    & \multicolumn{1}{c|}{0.00}     & \multicolumn{1}{c|}{0.00}     & 0.00     & \multicolumn{1}{c|}{9.02E-14} & \multicolumn{1}{c|}{3.18E-07} & \multicolumn{1}{c|}{1.30E-11} & \multicolumn{1}{c|}{4.71E-09} & 4.53E-11 \\ \hline
					Maximum        & \multicolumn{1}{c|}{10000}  & \multicolumn{1}{c|}{5850}    & \multicolumn{1}{c|}{7150}     & \multicolumn{1}{c|}{8650}     & 10000    & \multicolumn{1}{c|}{123.81} & \multicolumn{1}{c|}{55.80}   & \multicolumn{1}{c|}{55.17}    & \multicolumn{1}{c|}{79.53}    & 96.40    & \multicolumn{1}{c|}{3.10E-04} & \multicolumn{1}{c|}{9.86E-06} & \multicolumn{1}{c|}{9.91E-06} & \multicolumn{1}{c|}{9.94E-06} & 3.37E-04 \\ \hline
			\end{tabular}}
		\end{table}
		
		\begin{table}[htbp]
			\centering
			\rotatebox[origin=c]{90}{%
				\begin{varwidth}{\textheight}
					\caption{Numerical results for 25 QP problems selected from Maros-Mészáros collection. In the table, ``a", ``$b_{2}$", ``$b_{15}$", ``$b_{30}$", and ``$b_{45}$" denote the pADMM, and the acc-pADMM with $\alpha=2,15,30,45$, respectively.}
					\label{tab:my-table}
					\resizebox{1.0\textwidth}{!}{
						\renewcommand\arraystretch{1.24}
						\begin{tabular}{|c|c|c|c|ccccc|ccccc|ccccc|}
							\hline
							&          &       &       & \multicolumn{5}{c|}{Iter}                                                                                                            & \multicolumn{5}{c|}{Time (s)}                                                                                                         & \multicolumn{5}{c|}{$\mathrm{KKT}_{\mathrm{res}}$}                                                                                                          \\ \hline
							& problem  & m     & n     & \multicolumn{1}{c|}{a}     & \multicolumn{1}{c|}{$b_{2}$} & \multicolumn{1}{c|}{$b_{15}$} & \multicolumn{1}{c|}{$b_{30}$} & $b_{45}$ & \multicolumn{1}{c|}{a}      & \multicolumn{1}{c|}{$b_{2}$} & \multicolumn{1}{c|}{$b_{15}$} & \multicolumn{1}{c|}{$b_{30}$} & $b_{45}$ & \multicolumn{1}{c|}{a}       & \multicolumn{1}{c|}{$b_{2}$} & \multicolumn{1}{c|}{$b_{15}$} & \multicolumn{1}{c|}{$b_{30}$} & $b_{45}$ \\ \hline
							1  & AUG2D    & 9604  & 19404 & \multicolumn{1}{c|}{50}    & \multicolumn{1}{c|}{150}     & \multicolumn{1}{c|}{100}      & \multicolumn{1}{c|}{50}       & 50       & \multicolumn{1}{c|}{0.18}   & \multicolumn{1}{c|}{0.19}    & \multicolumn{1}{c|}{0.13}     & \multicolumn{1}{c|}{0.07}     & 0.09     & \multicolumn{1}{c|}{9.0E-14} & \multicolumn{1}{c|}{6.1E-06} & \multicolumn{1}{c|}{1.3E-11}  & \multicolumn{1}{c|}{4.6E-06}  & 1.7E-06  \\ \hline
							2  & AUG2DC   & 10000 & 20200 & \multicolumn{1}{c|}{50}    & \multicolumn{1}{c|}{150}     & \multicolumn{1}{c|}{100}      & \multicolumn{1}{c|}{50}       & 50       & \multicolumn{1}{c|}{0.08}   & \multicolumn{1}{c|}{0.19}    & \multicolumn{1}{c|}{0.12}     & \multicolumn{1}{c|}{0.07}     & 0.07     & \multicolumn{1}{c|}{9.4E-14} & \multicolumn{1}{c|}{6.2E-06} & \multicolumn{1}{c|}{1.3E-11}  & \multicolumn{1}{c|}{4.5E-06}  & 1.6E-06  \\ \hline
							3  & AUG3DQP  & 972   & 3133  & \multicolumn{1}{c|}{150}   & \multicolumn{1}{c|}{350}     & \multicolumn{1}{c|}{100}      & \multicolumn{1}{c|}{100}      & 100      & \multicolumn{1}{c|}{0.04}   & \multicolumn{1}{c|}{0.06}    & \multicolumn{1}{c|}{0.02}     & \multicolumn{1}{c|}{0.03}     & 0.02     & \multicolumn{1}{c|}{6.6E-08} & \multicolumn{1}{c|}{3.2E-07} & \multicolumn{1}{c|}{8.7E-06}  & \multicolumn{1}{c|}{4.1E-08}  & 3.5E-08  \\ \hline
							4  & CONT-101 & 9801  & 9900  & \multicolumn{1}{c|}{1850}  & \multicolumn{1}{c|}{1400}    & \multicolumn{1}{c|}{1400}     & \multicolumn{1}{c|}{1350}     & 1750     & \multicolumn{1}{c|}{5.62}   & \multicolumn{1}{c|}{4.26}    & \multicolumn{1}{c|}{4.25}     & \multicolumn{1}{c|}{4.14}     & 5.32     & \multicolumn{1}{c|}{9.1E-06} & \multicolumn{1}{c|}{6.8E-06} & \multicolumn{1}{c|}{4.8E-06}  & \multicolumn{1}{c|}{9.5E-06}  & 8.5E-06  \\ \hline
							5  & CONT-201 & 39601 & 39800 & \multicolumn{1}{c|}{1800}  & \multicolumn{1}{c|}{1600}    & \multicolumn{1}{c|}{1500}     & \multicolumn{1}{c|}{1500}     & 1650     & \multicolumn{1}{c|}{26.96}  & \multicolumn{1}{c|}{24.48}   & \multicolumn{1}{c|}{22.72}    & \multicolumn{1}{c|}{22.79}    & 25.39    & \multicolumn{1}{c|}{8.5E-06} & \multicolumn{1}{c|}{9.6E-06} & \multicolumn{1}{c|}{9.3E-06}  & \multicolumn{1}{c|}{9.9E-06}  & 8.1E-06  \\ \hline
							6  & CONT-300 & 89401 & 89700 & \multicolumn{1}{c|}{3250}  & \multicolumn{1}{c|}{1450}    & \multicolumn{1}{c|}{1450}     & \multicolumn{1}{c|}{2050}     & 2500     & \multicolumn{1}{c|}{123.81} & \multicolumn{1}{c|}{55.80}   & \multicolumn{1}{c|}{55.17}    & \multicolumn{1}{c|}{79.53}    & 96.40    & \multicolumn{1}{c|}{7.9E-06} & \multicolumn{1}{c|}{8.0E-06} & \multicolumn{1}{c|}{9.9E-06}  & \multicolumn{1}{c|}{9.8E-06}  & 9.6E-06  \\ \hline
							7  & GOULDQP3 & 349   & 699   & \multicolumn{1}{c|}{150}   & \multicolumn{1}{c|}{300}     & \multicolumn{1}{c|}{150}      & \multicolumn{1}{c|}{150}      & 150      & \multicolumn{1}{c|}{0.01}   & \multicolumn{1}{c|}{0.01}    & \multicolumn{1}{c|}{0.01}     & \multicolumn{1}{c|}{0.01}     & 0.01     & \multicolumn{1}{c|}{9.2E-07} & \multicolumn{1}{c|}{3.0E-06} & \multicolumn{1}{c|}{6.9E-06}  & \multicolumn{1}{c|}{5.0E-06}  & 4.5E-06  \\ \hline
							8  & HS118    & 29    & 44    & \multicolumn{1}{c|}{10000} & \multicolumn{1}{c|}{300}     & \multicolumn{1}{c|}{200}      & \multicolumn{1}{c|}{250}      & 250      & \multicolumn{1}{c|}{0.14}   & \multicolumn{1}{c|}{0.00}    & \multicolumn{1}{c|}{0.00}     & \multicolumn{1}{c|}{0.00}     & 0.00     & \multicolumn{1}{c|}{1.9E-04} & \multicolumn{1}{c|}{1.8E-06} & \multicolumn{1}{c|}{8.0E-07}  & \multicolumn{1}{c|}{2.6E-06}  & 8.4E-06  \\ \hline
							9  & KSIP     & 1000  & 1020  & \multicolumn{1}{c|}{600}   & \multicolumn{1}{c|}{650}     & \multicolumn{1}{c|}{550}      & \multicolumn{1}{c|}{550}      & 650      & \multicolumn{1}{c|}{0.49}   & \multicolumn{1}{c|}{0.53}    & \multicolumn{1}{c|}{0.45}     & \multicolumn{1}{c|}{0.46}     & 0.53     & \multicolumn{1}{c|}{1.0E-05} & \multicolumn{1}{c|}{7.9E-06} & \multicolumn{1}{c|}{9.0E-06}  & \multicolumn{1}{c|}{9.0E-06}  & 6.3E-06  \\ \hline
							10 & QRECIPE  & 59    & 116   & \multicolumn{1}{c|}{1100}  & \multicolumn{1}{c|}{1050}    & \multicolumn{1}{c|}{350}      & \multicolumn{1}{c|}{350}      & 350      & \multicolumn{1}{c|}{0.02}   & \multicolumn{1}{c|}{0.02}    & \multicolumn{1}{c|}{0.01}     & \multicolumn{1}{c|}{0.01}     & 0.01     & \multicolumn{1}{c|}{9.1E-06} & \multicolumn{1}{c|}{2.7E-06} & \multicolumn{1}{c|}{5.3E-07}  & \multicolumn{1}{c|}{4.7E-09}  & 4.5E-11  \\ \hline
							11 & QSCAGR25 & 274   & 473   & \multicolumn{1}{c|}{10000} & \multicolumn{1}{c|}{3800}    & \multicolumn{1}{c|}{4400}     & \multicolumn{1}{c|}{6350}     & 3700     & \multicolumn{1}{c|}{0.29}   & \multicolumn{1}{c|}{0.11}    & \multicolumn{1}{c|}{0.13}     & \multicolumn{1}{c|}{0.19}     & 0.11     & \multicolumn{1}{c|}{3.1E-04} & \multicolumn{1}{c|}{8.4E-06} & \multicolumn{1}{c|}{8.5E-06}  & \multicolumn{1}{c|}{7.8E-06}  & 8.9E-06  \\ \hline
							12 & QSCORPIO & 161   & 226   & \multicolumn{1}{c|}{1300}  & \multicolumn{1}{c|}{350}     & \multicolumn{1}{c|}{250}      & \multicolumn{1}{c|}{250}      & 250      & \multicolumn{1}{c|}{0.03}   & \multicolumn{1}{c|}{0.01}    & \multicolumn{1}{c|}{0.01}     & \multicolumn{1}{c|}{0.01}     & 0.01     & \multicolumn{1}{c|}{9.3E-06} & \multicolumn{1}{c|}{5.1E-06} & \multicolumn{1}{c|}{1.6E-06}  & \multicolumn{1}{c|}{1.6E-06}  & 1.6E-06  \\ \hline
							13 & QSCRS8   & 192   & 945   & \multicolumn{1}{c|}{9850}  & \multicolumn{1}{c|}{5850}    & \multicolumn{1}{c|}{7150}     & \multicolumn{1}{c|}{8650}     & 10000    & \multicolumn{1}{c|}{0.34}   & \multicolumn{1}{c|}{0.20}    & \multicolumn{1}{c|}{0.25}     & \multicolumn{1}{c|}{0.31}     & 0.34     & \multicolumn{1}{c|}{9.5E-06} & \multicolumn{1}{c|}{8.1E-06} & \multicolumn{1}{c|}{9.5E-06}  & \multicolumn{1}{c|}{8.8E-06}  & 3.4E-04  \\ \hline
							14 & QSCSD1   & 77    & 760   & \multicolumn{1}{c|}{1050}  & \multicolumn{1}{c|}{500}     & \multicolumn{1}{c|}{250}      & \multicolumn{1}{c|}{300}      & 300      & \multicolumn{1}{c|}{0.04}   & \multicolumn{1}{c|}{0.02}    & \multicolumn{1}{c|}{0.01}     & \multicolumn{1}{c|}{0.01}     & 0.01     & \multicolumn{1}{c|}{6.2E-06} & \multicolumn{1}{c|}{6.6E-06} & \multicolumn{1}{c|}{2.3E-06}  & \multicolumn{1}{c|}{2.6E-06}  & 2.0E-07  \\ \hline
							15 & QSCSD8   & 397   & 2750  & \multicolumn{1}{c|}{2650}  & \multicolumn{1}{c|}{750}     & \multicolumn{1}{c|}{900}      & \multicolumn{1}{c|}{900}      & 950      & \multicolumn{1}{c|}{0.26}   & \multicolumn{1}{c|}{0.08}    & \multicolumn{1}{c|}{0.09}     & \multicolumn{1}{c|}{0.09}     & 0.10     & \multicolumn{1}{c|}{9.4E-06} & \multicolumn{1}{c|}{1.6E-06} & \multicolumn{1}{c|}{6.1E-07}  & \multicolumn{1}{c|}{7.7E-07}  & 5.0E-06  \\ \hline
							16 & QSCTAP2  & 977   & 2303  & \multicolumn{1}{c|}{10000} & \multicolumn{1}{c|}{450}     & \multicolumn{1}{c|}{600}      & \multicolumn{1}{c|}{1650}     & 2500     & \multicolumn{1}{c|}{0.87}   & \multicolumn{1}{c|}{0.05}    & \multicolumn{1}{c|}{0.06}     & \multicolumn{1}{c|}{0.15}     & 0.23     & \multicolumn{1}{c|}{1.6E-04} & \multicolumn{1}{c|}{3.4E-06} & \multicolumn{1}{c|}{6.2E-06}  & \multicolumn{1}{c|}{9.4E-06}  & 8.6E-06  \\ \hline
							17 & QSCTAP3  & 1274  & 3041  & \multicolumn{1}{c|}{4200}  & \multicolumn{1}{c|}{450}     & \multicolumn{1}{c|}{500}      & \multicolumn{1}{c|}{700}      & 1000     & \multicolumn{1}{c|}{0.46}   & \multicolumn{1}{c|}{0.06}    & \multicolumn{1}{c|}{0.06}     & \multicolumn{1}{c|}{0.09}     & 0.12     & \multicolumn{1}{c|}{6.2E-06} & \multicolumn{1}{c|}{1.8E-06} & \multicolumn{1}{c|}{8.5E-06}  & \multicolumn{1}{c|}{4.9E-06}  & 5.3E-06  \\ \hline
							18 & QSHIP04L & 288   & 1901  & \multicolumn{1}{c|}{650}   & \multicolumn{1}{c|}{350}     & \multicolumn{1}{c|}{300}      & \multicolumn{1}{c|}{300}      & 300      & \multicolumn{1}{c|}{0.04}   & \multicolumn{1}{c|}{0.02}    & \multicolumn{1}{c|}{0.02}     & \multicolumn{1}{c|}{0.02}     & 0.02     & \multicolumn{1}{c|}{7.3E-06} & \multicolumn{1}{c|}{2.6E-06} & \multicolumn{1}{c|}{3.9E-06}  & \multicolumn{1}{c|}{4.6E-06}  & 5.7E-06  \\ \hline
							19 & QSHIP04S & 188   & 1253  & \multicolumn{1}{c|}{7350}  & \multicolumn{1}{c|}{300}     & \multicolumn{1}{c|}{250}      & \multicolumn{1}{c|}{300}      & 300      & \multicolumn{1}{c|}{0.27}   & \multicolumn{1}{c|}{0.01}    & \multicolumn{1}{c|}{0.01}     & \multicolumn{1}{c|}{0.01}     & 0.01     & \multicolumn{1}{c|}{9.3E-06} & \multicolumn{1}{c|}{5.5E-06} & \multicolumn{1}{c|}{1.3E-06}  & \multicolumn{1}{c|}{2.8E-06}  & 5.9E-06  \\ \hline
							20 & QSHIP08L & 478   & 3137  & \multicolumn{1}{c|}{950}   & \multicolumn{1}{c|}{300}     & \multicolumn{1}{c|}{250}      & \multicolumn{1}{c|}{250}      & 250      & \multicolumn{1}{c|}{0.18}   & \multicolumn{1}{c|}{0.06}    & \multicolumn{1}{c|}{0.06}     & \multicolumn{1}{c|}{0.06}     & 0.05     & \multicolumn{1}{c|}{1.0E-05} & \multicolumn{1}{c|}{9.9E-06} & \multicolumn{1}{c|}{7.2E-06}  & \multicolumn{1}{c|}{4.7E-06}  & 8.7E-06  \\ \hline
							21 & QSHIP08S & 256   & 1578  & \multicolumn{1}{c|}{1300}  & \multicolumn{1}{c|}{300}     & \multicolumn{1}{c|}{200}      & \multicolumn{1}{c|}{250}      & 250      & \multicolumn{1}{c|}{0.09}   & \multicolumn{1}{c|}{0.02}    & \multicolumn{1}{c|}{0.02}     & \multicolumn{1}{c|}{0.02}     & 0.02     & \multicolumn{1}{c|}{8.9E-06} & \multicolumn{1}{c|}{3.1E-06} & \multicolumn{1}{c|}{5.0E-06}  & \multicolumn{1}{c|}{3.0E-06}  & 7.2E-06  \\ \hline
							22 & QSHIP12L & 637   & 4226  & \multicolumn{1}{c|}{2150}  & \multicolumn{1}{c|}{450}     & \multicolumn{1}{c|}{300}      & \multicolumn{1}{c|}{350}      & 350      & \multicolumn{1}{c|}{0.52}   & \multicolumn{1}{c|}{0.12}    & \multicolumn{1}{c|}{0.09}     & \multicolumn{1}{c|}{0.10}     & 0.10     & \multicolumn{1}{c|}{8.4E-06} & \multicolumn{1}{c|}{2.6E-06} & \multicolumn{1}{c|}{9.4E-06}  & \multicolumn{1}{c|}{4.2E-06}  & 5.5E-06  \\ \hline
							23 & QSHIP12S & 322   & 1953  & \multicolumn{1}{c|}{1750}  & \multicolumn{1}{c|}{550}     & \multicolumn{1}{c|}{400}      & \multicolumn{1}{c|}{400}      & 400      & \multicolumn{1}{c|}{0.14}   & \multicolumn{1}{c|}{0.05}    & \multicolumn{1}{c|}{0.04}     & \multicolumn{1}{c|}{0.04}     & 0.04     & \multicolumn{1}{c|}{7.3E-06} & \multicolumn{1}{c|}{3.4E-06} & \multicolumn{1}{c|}{4.0E-06}  & \multicolumn{1}{c|}{6.0E-06}  & 6.9E-06  \\ \hline
							24 & QSIERRA  & 915   & 2347  & \multicolumn{1}{c|}{1300}  & \multicolumn{1}{c|}{700}     & \multicolumn{1}{c|}{550}      & \multicolumn{1}{c|}{550}      & 600      & \multicolumn{1}{c|}{0.09}   & \multicolumn{1}{c|}{0.06}    & \multicolumn{1}{c|}{0.05}     & \multicolumn{1}{c|}{0.05}     & 0.05     & \multicolumn{1}{c|}{9.0E-06} & \multicolumn{1}{c|}{7.7E-06} & \multicolumn{1}{c|}{7.9E-06}  & \multicolumn{1}{c|}{8.1E-06}  & 9.0E-06  \\ \hline
							25 & QSTANDAT & 192   & 500   & \multicolumn{1}{c|}{3650}  & \multicolumn{1}{c|}{850}     & \multicolumn{1}{c|}{600}      & \multicolumn{1}{c|}{600}      & 600      & \multicolumn{1}{c|}{0.10}   & \multicolumn{1}{c|}{0.03}    & \multicolumn{1}{c|}{0.02}     & \multicolumn{1}{c|}{0.02}     & 0.02     & \multicolumn{1}{c|}{9.7E-06} & \multicolumn{1}{c|}{4.3E-06} & \multicolumn{1}{c|}{6.6E-06}  & \multicolumn{1}{c|}{3.2E-06}  & 2.6E-06  \\ \hline
						\end{tabular}
					}
			\end{varwidth}}
		\end{table}
		
		\begin{figure}[htbp]
			\centering
			\includegraphics[width=0.6\textwidth]{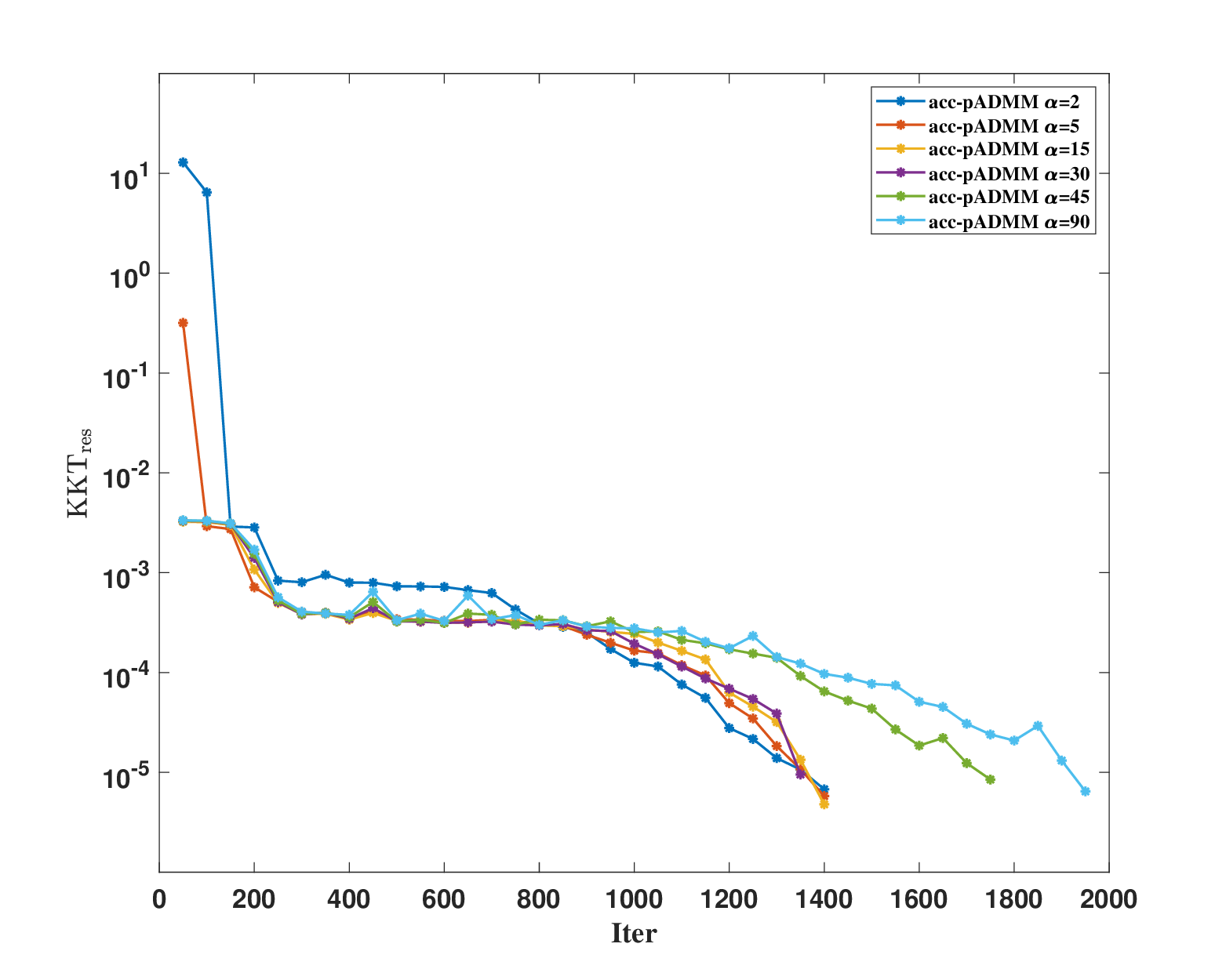}
			\caption{The ${\rm KKT_{res}}$ for problem 4 (CONT-101) obtained by the acc-pADMM with different $\alpha$.}
			\label{Fig:KKT_alpha}
		\end{figure}		
		To further verify the complexity results outlined in Theorem \ref{Th:complexity-acc-pADMM}, we utilize problem 11 listed in Table \ref{tab:my-table} as an illustrative example. In this experiment, we fix $\sigma=1$ and do not employ the restart strategy. The numerical results of ${\rm KKT_{res}}$ and the relative primal function value gap denoted by $|h(y, z)|/(f_1(y^*)+f_2(z^*))$ are presented in Figure \ref{Fig:KKT_alpha3}. Figure \ref{Fig:KKT_alpha3} illustrates that the acc-pADMM achieves non-asymptotic $O(1/k)$ convergence rates when $\alpha=2$, and it attains asymptotic $o(1/k)$ convergence rates when $\alpha>2$. This numerical result is in concordance with the theoretical findings presented in Theorem \ref{Th:complexity-acc-pADMM}.
		\begin{figure}[htbp]
			\centering
			\includegraphics[width=0.6\textwidth]{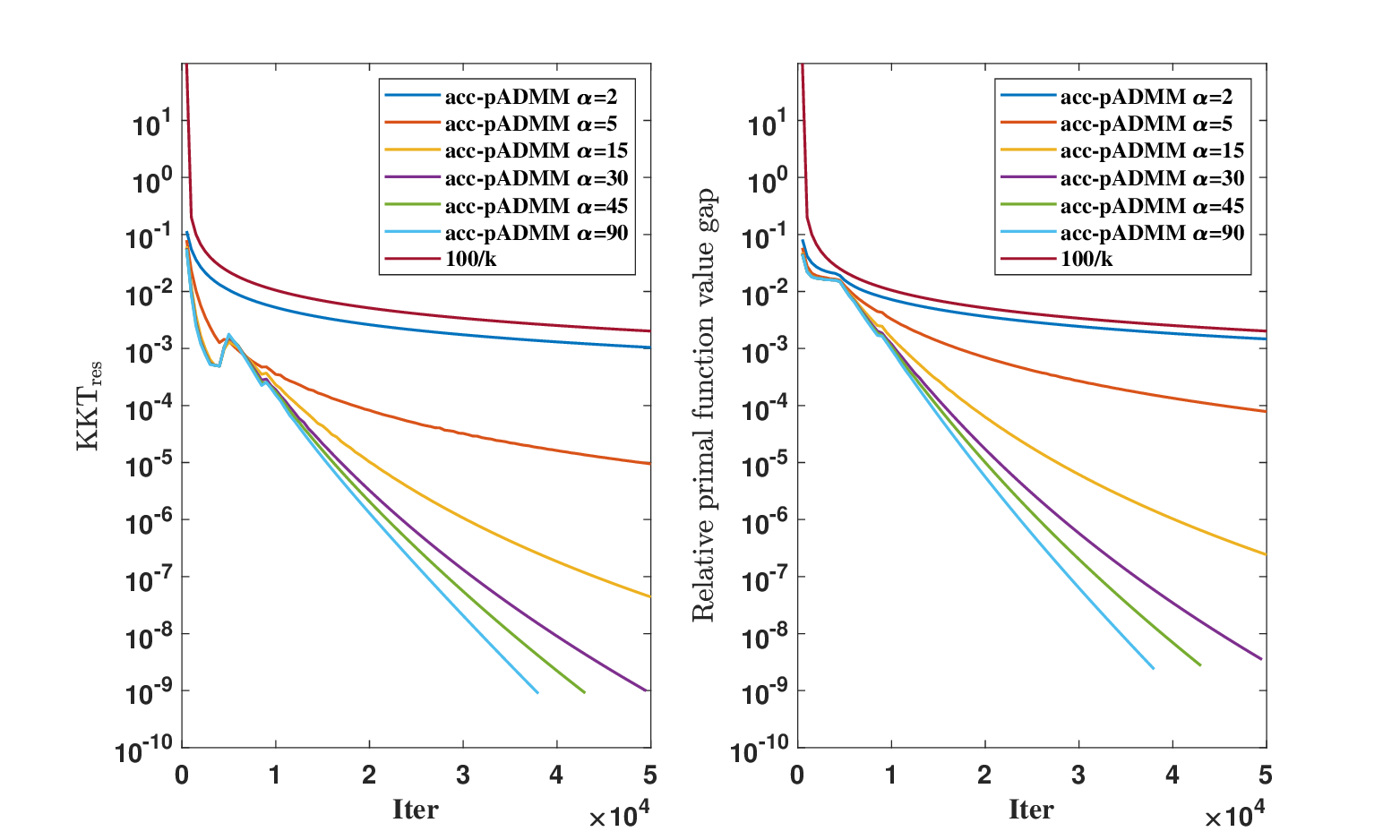}
			\caption{The ${\rm KKT_{res}}$ and relative primal function value gap for problem 11 (QSCAGR25) obtained by the acc-pADMM with different $\alpha$.}
			\label{Fig:KKT_alpha3}
		\end{figure}
		\section{Conclusion}\label{Sec: conclusion}
		In this paper, we proposed an accelerated dPPM with both asymptotic $o(1/k)$ and non-asymptotic $O(1/k)$ convergence rates by unifying the Halpern iteration and the fast Krasnosel'ski\u{i}-Mann iteration. Leveraging the equivalence between the pADMM and the dPPM, we derived an accelerated pADMM, which exhibited both asymptotic $o(1/k)$ and non-asymptotic $O(1/k)$ convergence rates with respect to the KKT residual and the primal objective function value gap. Numerical experiments demonstrated the superior performance of the accelerated pADMM over the pADMM when solving convex {QP} problems. Recently, Kong and Monteiro \cite{kong2024global} have introduced a dampened proximal ADMM for solving linearly constrained nonseparable nonconvex optimization problems, which can obtain an {approximate} first-order stationary point within $O(\varepsilon^{-3})$ iterations for a given tolerance $\varepsilon>0$ under a basic Slater point condition. As {a future research direction}, it {would be}  interesting to extend the acceleration techniques presented in this paper to devise a novel accelerated pADMM variant for tackling nonconvex optimization problems with a better complexity result than $O(\varepsilon^{-3})$.
		
	   {After the announcement of the first version of this paper, there are several new developments that deserve to be mentioned. In particular, an implementation of our accelerated pADMM based on the Halpern iteration, referred to as HPR-LP \cite{chen2024hpr}, has shown promising performance in solving large-scale linear programming problems on GPUs, outperforming the award-winning solver PDLP \cite{applegate2021practical, applegate2023faster, lu2023cupdlp}. This underscores the great potential of the accelerated pADMM for efficiently addressing large-scale convex optimization problems. More recently, Boţ et al. \cite{boct2024generalized} derived a generalized fast KM method with more flexible parameters. It will be very interesting to see how this method can be incorporated into the accelerated pADMM framework to further improve its practical performance.}

		\bibliographystyle{siamplain}
		\bibliography{references}

\begin{thebibliography}{10}

\bibitem{applegate2021practical}
{\sc D.~Applegate, M.~D{\'\i}az, O.~Hinder, H.~Lu, M.~Lubin, B.~O'Donoghue, and
  W.~Schudy}, {\em Practical large-scale linear programming using primal-dual
  hybrid gradient}, Advances in Neural Information Processing Systems, 34
  (2021), pp.~20243--20257.

\bibitem{applegate2023faster}
{\sc D.~Applegate, O.~Hinder, H.~Lu, and M.~Lubin}, {\em Faster first-order
  primal-dual methods for linear programming using restarts and sharpness},
  Mathematical Programming, 201 (2023), pp.~133--184.

\bibitem{bauschke2017convex}
{\sc H.~H. Bauschke and P.~L. Combettes}, {\em Convex Analysis and Monotone
  Operator Theory in Hilbert Spaces,2nd edn.}, Springer, New York, 2017.

\bibitem{beck2009fast}
{\sc A.~Beck and M.~Teboulle}, {\em A fast iterative shrinkage-thresholding
  algorithm for linear inverse problems}, SIAM J. Imaging Sci., 2 (2009),
  pp.~183--202.

\bibitem{bonnans1995family}
{\sc J.~F. Bonnans, J.~C. Gilbert, C.~Lemar{\'e}chal, and C.~A.
  Sagastiz{\'a}bal}, {\em A family of variable metric proximal methods}, Math.
  Program., 68 (1995), pp.~15--47.

\bibitem{boct2024generalized}
{\sc R.~I. Bo{\c{t}}, E.~Chenchene, and J.~M. Fadili}, {\em Generalized fast
  {K}rasnosel'ski\u i-{M}ann method with preconditioners}, arXiv preprint
  arXiv:2411.18574,  (2024).

\bibitem{bot2023fastogda}
{\sc R.~I. Bo{\c{t}}, E.~R. Csetnek, and D.-K. Nguyen}, {\em Fast optimistic
  gradient descent ascent ({OGDA}) method in continuous and discrete time},
  Found. Comput. Math.,  (2023), pp.~1--60.

\bibitem{bot2023fast}
{\sc R.~I. Bo{\c t} and D.-K. Nguyen}, {\em Fast {K}rasnosel'ski\u i-{M}ann
  algorithm with a convergence rate of the fixed point iteration of
  {$o(\frac{1}{ k})$}}, SIAM J. Numer. Anal., 61 (2023), pp.~2813--2843.

\bibitem{bredies2022degenerate}
{\sc K.~Bredies, E.~Chenchene, D.~A. Lorenz, and E.~Naldi}, {\em Degenerate
  preconditioned proximal point algorithms}, SIAM J. Optim., 32 (2022),
  pp.~2376--2401.

\bibitem{bredies2015preconditioned}
{\sc K.~Bredies and H.~Sun}, {\em Preconditioned {D}ouglas--{R}achford
  splitting methods for convex-concave saddle-point problems}, SIAM J. Numer.
  Anal., 53 (2015), pp.~421--444.

\bibitem{bredies2017proximal}
{\sc K.~Bredies and H.~Sun}, {\em A proximal point analysis of the
  preconditioned alternating direction method of multipliers}, J. Optim. Theory
  Appl., 173 (2017), pp.~878--907.

\bibitem{brezis1978produits}
{\sc H.~Br{\'e}zis and P.~L. Lions}, {\em Produits infinis de r{\'e}solvantes},
  Israel J. Math., 29 (1978), pp.~329--345.

\bibitem{chambolle2011first}
{\sc A.~Chambolle and T.~Pock}, {\em A first-order primal-dual algorithm for
  convex problems with applications to imaging}, J. Math. Imaging Vision, 40
  (2011), pp.~120--145.

\bibitem{chen2015inertial}
{\sc C.~Chen, R.~H. Chan, S.~Ma, and J.~Yang}, {\em Inertial proximal {ADMM}
  for linearly constrained separable convex optimization}, SIAM J. Imaging
  Sci., 8 (2015), pp.~2239--2267.

\bibitem{chen2024hpr}
{\sc K.~Chen, D.~Sun, Y.~Yuan, G.~Zhang, and X.~Zhao}, {\em {HPR-LP}: {A}n
  implementation of an {HPR} method for solving linear programming}, arXiv
  preprint arXiv:2408.12179,  (2024).

\bibitem{condat2013primal}
{\sc L.~Condat}, {\em A primal--dual splitting method for convex optimization
  involving lipschitzian, proximable and linear composite terms}, J. Optim.
  Theory Appl., 158 (2013), pp.~460--479.

\bibitem{contreras2023optimal}
{\sc J.~P. Contreras and R.~Cominetti}, {\em Optimal error bounds for
  non-expansive fixed-point iterations in normed spaces}, Math. Program., 199
  (2023), pp.~343--374.

\bibitem{cui2016convergence}
{\sc Y.~Cui, X.~Li, D.~F. Sun, and K.-C. Toh}, {\em On the convergence
  properties of a majorized alternating direction method of multipliers for
  linearly constrained convex optimization problems with coupled objective
  functions}, J. Optim. Theory Appl., 169 (2016), pp.~1013--1041.

\bibitem{davis2016convergence}
{\sc D.~Davis and W.~Yin}, {\em Convergence rate analysis of several splitting
  schemes}, in Splitting Methods in Communication, Imaging, Science, and
  Engineering, R.~Glowinski, S.~J. Osher, and W.~Yin, eds., Springer, 2016,
  pp.~115--163.

\bibitem{eckstein1994some}
{\sc J.~Eckstein}, {\em Some saddle-function splitting methods for convex
  programming}, Optim. Methods Softw., 4 (1994), pp.~75--83.

\bibitem{eckstein1992douglas}
{\sc J.~Eckstein and D.~P. Bertsekas}, {\em On the {D}ouglas-{R}achford
  splitting method and the proximal point algorithm for maximal monotone
  operators}, Math. Program., 55 (1992), pp.~293--318.

\bibitem{fazel2013hankel}
{\sc M.~Fazel, T.~K. Pong, D.~F. Sun, and P.~Tseng}, {\em Hankel matrix rank
  minimization with applications to system identification and realization},
  SIAM J. Matrix Anal. Appl., 34 (2013), pp.~946--977.

\bibitem{gabay1976dual}
{\sc D.~Gabay and B.~Mercier}, {\em A dual algorithm for the solution of
  nonlinear variational problems via finite element approximation}, Comput.
  Math. Appl., 2 (1976), pp.~17--40.

\bibitem{glowinski1975approximation}
{\sc R.~Glowinski and A.~Marroco}, {\em Sur l'approximation, par
  {\'e}l{\'e}ments finis d'ordre un, et la r{\'e}solution, par
  p{\'e}nalisation-dualit{\'e} d'une classe de probl{\`e}mes de dirichlet non
  lin{\'e}aires}, Revue fran{\c{c}}aise d'automatique, informatique, recherche
  op{\'e}rationnelle. Analyse num{\'e}rique, 9 (1975), pp.~41--76.

\bibitem{goldstein2014fast}
{\sc T.~Goldstein, B.~O'Donoghue, S.~Setzer, and R.~Baraniuk}, {\em Fast
  alternating direction optimization methods}, SIAM J. Imaging Sci., 7 (2014),
  pp.~1588--1623.

\bibitem{halpern1967fixed}
{\sc B.~Halpern}, {\em Fixed points of nonexpanding maps}, Bull. Amer. Math.
  Soc., 73 (1967), pp.~957--961.

\bibitem{han2018linear}
{\sc D.~Han, D.~F. Sun, and L.~Zhang}, {\em Linear rate convergence of the
  alternating direction method of multipliers for convex composite
  programming}, Math. Oper. Res., 43 (2018), pp.~622--637.

\bibitem{kim2021accelerated}
{\sc D.~Kim}, {\em Accelerated proximal point method for maximally monotone
  operators}, Math. Program., 190 (2021), pp.~57--87.

\bibitem{kong2024global}
{\sc W.~Kong and R.~D. Monteiro}, {\em Global complexity bound of a proximal
  {A}{D}{M}{M} for linearly constrained nonseparable nonconvex composite
  programming}, SIAM J. Optim., 34 (2024), pp.~201--224.

\bibitem{li2019accelerated}
{\sc H.~Li and Z.~Lin}, {\em Accelerated alternating direction method of
  multipliers: {A}n optimal {$O(1/K)$} nonergodic analysis}, J. Sci. Comput.,
  79 (2019), pp.~671--699.

\bibitem{li2016schur}
{\sc X.~Li, D.~F. Sun, and K.-C. Toh}, {\em A {S}chur complement based
  semi-proximal {A}{D}{M}{M} for convex quadratic conic programming and
  extensions}, Math. Program., 155 (2016), pp.~333--373.

\bibitem{li2018qsdpnal}
{\sc X.~Li, D.~F. Sun, and K.-C. Toh}, {\em {Q}{S}{D}{P}{N}{A}{L}: {A}
  two-phase augmented {L}agrangian method for convex quadratic semidefinite
  programming}, Math. Program. Comput., 10 (2018), pp.~703--743.

\bibitem{li2019block}
{\sc X.~Li, D.~F. Sun, and K.-C. Toh}, {\em A block symmetric {G}auss--{S}eidel
  decomposition theorem for convex composite quadratic programming and its
  applications}, Math. Program., 175 (2019), pp.~395--418.

\bibitem{li2020asymptotically}
{\sc X.~Li, D.~F. Sun, and K.-C. Toh}, {\em An asymptotically superlinearly
  convergent semismooth {N}ewton augmented {L}agrangian method for linear
  programming}, SIAM J. Optim., 30 (2020), pp.~2410--2440.

\bibitem{liang2022qppal}
{\sc L.~Liang, X.~Li, D.~F. Sun, and K.-C. Toh}, {\em {Q}{P}{P}{A}{L}: {A}
  two-phase proximal augmented {L}agrangian method for high-dimensional convex
  quadratic programming problems}, ACM Trans. Math. Software, 48 (2022),
  pp.~1--27.

\bibitem{lieder2021convergence}
{\sc F.~Lieder}, {\em On the convergence rate of the {H}alpern-iteration},
  Optim. Lett., 15 (2021), pp.~405--418.

\bibitem{lions1979splitting}
{\sc P.-L. Lions and B.~Mercier}, {\em Splitting algorithms for the sum of two
  nonlinear operators}, SIAM J. Numer. Anal., 16 (1979), pp.~964--979.

\bibitem{lu2023cupdlp}
{\sc H.~Lu and J.~Yang}, {\em cu{PDLP.jl}: {A} {GPU} implementation of
  restarted primal-dual hybrid gradient for linear programming in julia}, arXiv
  preprint arXiv:2311.12180,  (2023).

\bibitem{maros1999repository}
{\sc I.~Maros and C.~M{\'e}sz{\'a}ros}, {\em A repository of convex quadratic
  programming problems}, Optim. Methods Softw., 11 (1999), pp.~671--681.

\bibitem{monteiro2013iteration}
{\sc R.~D. Monteiro and B.~F. Svaiter}, {\em Iteration-complexity of
  block-decomposition algorithms and the alternating direction method of
  multipliers}, SIAM J. Optim., 23 (2013), pp.~475--507.

\bibitem{nesterov1983method}
{\sc Y.~Nesterov}, {\em A method for unconstrained convex minimization problem
  with the rate of convergence {O}(1/$k^2$)}, Dokl. Akad. Nauk. SSSR, 269
  (1983), p.~543.

\bibitem{ouyang2015accelerated}
{\sc Y.~Ouyang, Y.~Chen, G.~Lan, and E.~Pasiliao~Jr}, {\em An accelerated
  linearized alternating direction method of multipliers}, SIAM J. Imaging
  Sci., 8 (2015), pp.~644--681.

\bibitem{polyak1987introduction}
{\sc B.~T. Polyak}, {\em Introduction to {O}ptimization}, Optimization
  Software, New York, 1987.

\bibitem{rockafellar1970convex}
{\sc R.~T. Rockafellar}, {\em Convex {A}nalysis}, Princeton {U}niversity
  {P}ress, Princeton, NJ, 1970.

\bibitem{rockafellar1976augmented}
{\sc R.~T. Rockafellar}, {\em Augmented {L}agrangians and applications of the
  proximal point algorithm in convex programming}, Math. Oper. Res., 1 (1976),
  pp.~97--116.

\bibitem{rockafellar1976monotone}
{\sc R.~T. Rockafellar}, {\em Monotone operators and the proximal point
  algorithm}, SIAM J. Control Optim., 14 (1976), pp.~877--898.

\bibitem{rockafellar1998variational}
{\sc R.~T. Rockafellar and R.~J.-B. Wets}, {\em {V}ariational {A}nalysis},
  Springer, New York, 1998.

\bibitem{sabach2017first}
{\sc S.~Sabach and S.~Shtern}, {\em A first order method for solving convex
  bilevel optimization problems}, SIAM J. Optim., 27 (2017), pp.~640--660.

\bibitem{sabach2022faster}
{\sc S.~Sabach and M.~Teboulle}, {\em Faster {L}agrangian-based methods in
  convex optimization}, SIAM J. Optim., 32 (2022), pp.~204--227.

\bibitem{tang2023self}
{\sc T.~Tang and K.-C. Toh}, {\em Self-adaptive {A}{D}{M}{M} for semi-strongly
  convex problems}, Math. Program. Comput.,  (2023), pp.~1--38.

\bibitem{tran2023halpern}
{\sc Q.~Tran-Dinh}, {\em From {H}alpern’s fixed-point iterations to
  {N}esterov’s accelerated interpretations for root-finding problems},
  Comput. Optim. Appl.,  (2023), pp.~1--38.

\bibitem{tran2021halpern}
{\sc Q.~Tran-Dinh and Y.~Luo}, {\em Halpern-type accelerated and splitting
  algorithms for monotone inclusions}, arXiv preprint arXiv:2110.08150,
  (2021).

\bibitem{wittmann1992approximation}
{\sc R.~Wittmann}, {\em Approximation of fixed points of nonexpansive
  mappings}, Arch. Math., 58 (1992), pp.~486--491.

\bibitem{xiao2018generalized}
{\sc Y.~Xiao, L.~Chen, and D.~Li}, {\em A generalized alternating direction
  method of multipliers with semi-proximal terms for convex composite conic
  programming}, Math. Program. Comput., 10 (2018), pp.~533--555.

\bibitem{xu2017accelerated}
{\sc Y.~Xu}, {\em Accelerated first-order primal-dual proximal methods for
  linearly constrained composite convex programming}, SIAM J. Optim., 27
  (2017), pp.~1459--1484.

\bibitem{yang2023accelerated}
{\sc B.~Yang, X.~Zhao, X.~Li, and D.~F. Sun}, {\em An accelerated proximal
  alternating direction method of multipliers for optimal decentralized control
  of uncertain systems}, arXiv preprint arXiv:2304.11037. To appear in Journal
  of Optimization Theory and Applications,  (2023).

\bibitem{zhang2022efficient}
{\sc G.~Zhang, Y.~Yuan, and D.~F. Sun}, {\em An efficient {H}{P}{R} algorithm
  for the {W}asserstein barycenter problem with ${O}$({D}im({P})/$\varepsilon$)
  computational complexity}, arXiv preprint arXiv:2211.14881,  (2022).

\end{thebibliography}
	\end{document}